\documentclass{cpamart1}     %% Input Standard CPAM class.
\received{Month 200X}       %\revised{date of revision}
\volume{000}
\startingpage{1}                      

\authorheadline{C. Mukherjee and S.R.S. Varadhan}
\titleheadline{Identification of the Polaron measure and its CLT}

\usepackage{mathptmx}

\usepackage{color}
\definecolor{Red}{rgb}{1,0,0}

\newtheorem{theorem}{Theorem}[section]
\newtheorem{lemma}[theorem]{Lemma}
\newtheorem{corollary}[theorem]{Corollary}

\theoremstyle{definition}

\theoremstyle{remark}
\newtheorem{remark}[theorem]{Remark}

 \renewcommand{\theequation}{\thesection.\arabic{equation}}
\makeatletter\@addtoreset{equation}{section}\makeatother

 \newfont{\bfit}{cmbxti10 scaled 1200}

\renewcommand{\d}{{\rm d}}

\newcommand{\cX}{{\mathcal X}}
\newcommand{\cY}{{\mathcal Y}}

 \newcommand{\e}{{\rm e} }

 \newcommand{\eps}{\varepsilon}
 \newcommand{\supp}{{\rm supp}}

 \newcommand{\R}{\mathbb{R}}
 \newcommand{\N}{\mathbb{N}}
 \newcommand{\Z}{\mathbb{Z}}

 \newcommand{\E}{\mathbb{E}}
 \renewcommand{\P}{\mathbb{P}}
 \def\1{{\mathchoice {1\mskip-4mu\mathrm l} 
{1\mskip-4mu\mathrm l}
{1\mskip-4.5mu\mathrm l} {1\mskip-5mu\mathrm l}}}

 \newcommand{\Mcal}{{\mathcal M}}

\newcommand{\heap}[2]{\genfrac{}{}{0pt}{}{#1}{#2}}

\newcommand{\ssup}[1] {{\scriptscriptstyle{({#1}})}}

\renewcommand{\labelenumi}{(\roman{enumi})}
\renewcommand{\subsection}{\secdef \subsct\sbsect}
\newcommand{\subsct}[2][default]{\refstepcounter{subsection}
\vspace{0.15cm}
{\flushleft\bf \arabic{section}.\arabic{subsection}~\bf #1  }
\nopagebreak\nopagebreak}
\newcommand{\sbsect}[1]{\vspace{0.1cm}\noindent
{\bf #1}\vspace{0.1cm}}

{\nopagebreak {\hfill\rule{2mm}{2mm}}\\ }

\begin{document}                        %% Standard LaTeX command

%%      -----------------------------------------------------------------------
%%      -------------------------------- TITLE -----------------------------
%%      -----------------------------------------------------------------------

\title{{{Identification of the Polaron measure I: Fixed coupling regime and the central limit theorem for large times}} }

%%      -----------------------------------------------------------------------------
%%      ------------------------------- AUTHORS -----------------------------
%%      ----------------------------------------------------------------------------
\author{Chiranjib Mukherjee}{University of M\"unster}
% EXAMPLE: \author{Bart Simpson}{Universit‰ Paris-Sorbonne (Paris IV)}
% Uncomment and fill in the following lines as needed
\author{S.R.S.Varadhan}{Courant Institute of Mathematical Sciences}
%\author{*** THIRD AUTHOR'S NAME ***}{*** THIRD AUTHOR'S AFFILIATION WHEN ARTICLE WAS WRITTEN ***}
%\author{*** FOURTH AUTHOR'S NAME ***}{*** FOURTH AUTHOR'S AFFILIATION WHEN ARTICLE WAS WRITTEN ***}
%\author{*** FIFTH AUTHOR'S NAME ***}{*** FIFTH AUTHOR'S AFFILIATION WHEN ARTICLE WAS WRITTEN ***}
% Add additional names and affiliations as necessary using above format
%%      ---------------------------------------------------------------------
%%      --------------------------- DEDICATION  (OPTIONAL)------------------- 
%%      ---------------------------------------------------------------------

%       Uncomment the following line to insert a dedication.

%\dedication{ *** DEDICATION *** }        %% Enter dedication between braces.

%%      ------------------------------------------------------------------------------------
%%      --------------------------- ABSTRACT (OPTIONAL)----------------------
%%      ------------------------------------------------------------------------------------

%% ***** UNCOMMENT THE FOLLOWING TO INSERT AN ABSTRACT *****

\begin{abstract}
We consider the Fr\"ohlich model of the {\it{Polaron}} whose path integral formulation leads to the transformed path measure
  $$
  \widehat{\mathbb P}_{\alpha,T}(\d\omega)= Z_{\alpha,T}^{-1}\,\, \exp\bigg\{\frac{\alpha}{2}\int_{-T}^T\int_{-T}^T\frac{e^{-|t-s|}}{|\omega(t)-\omega(s)|} \, d s \, d t\bigg\}\,\mathbb P(\d\omega)
  $$
  with respect to $\mathbb P$ which governs the law of the increments of the three dimensional Brownian motion on a finite interval $[-T,T]$, and $ Z_{\alpha,T}$ is the partition function or the normalizing constant and $\alpha>0$ is a constant. The Polaron measure reflects a self attractive interaction. According to a conjecture of Pekar that was proved in \cite{DV83}
  $$
  g_0=\lim_{\alpha \to\infty}\frac{1}{\alpha^2}\bigg[\lim_{T\to\infty}\frac{\log Z_{\alpha,T}}{2T}\bigg]
  $$
  exists and has a variational formula. In this article we show that for any $\alpha>0$, the infinite-volume limit $\widehat{\mathbb P}_{\alpha}=\lim_{T\to\infty}\widehat{\mathbb P}_{\alpha,T}$ exists which is also identified explicitly. As a corollary, we deduce the central limit theorem (for any $\alpha>0$ and as $T\to\infty$) for the distribution of $\frac{\omega(T)-\omega(-T)}{\sqrt{2T}}$ both under the finite-volume Polaron measure $\widehat{\mathbb P}_{\alpha,T}$ and its infinite-volume counterpart $\widehat{\mathbb P}_\alpha$, and obtain an expression for the limiting variance.

\end{abstract}

% With AMS-LaTeX, \maketitle follows the abstract
\maketitle   

%%      ---------------------------------------------------------------------
%%      ------------------- TABLE OF CONTENTS (OPTIONAL) --------------------
%%      ---------------------------------------------------------------------

%% ***** IF YOUR PAPER IS OVER 40 PAGES AND YOU WISH TO HAVE A TABLE
%% ***** OF CONTENTS, PLEASE UNCOMMENT THE FOLLOWING LINE

% \tableofcontents

%%      ---------------------------------------------------------------------
%%      ---------------------------- BODY OF PAPER --------------------------
%%      ---------------------------------------------------------------------

%%      Please input or insert the body of your paper here.

{{\section{Motivation and physical background of the Polaron.}}}

The  {\it{Polaron problem}}  in quantum mechanics is inspired by studying the slow movement 
of a charged particle, e.g. an electron, in a crystal whose lattice sites are polarized by this slow motion. The electron then drags around it a cloud of polarized lattice 
points which influences and determines the {\it{effective behavior}} of the electron. In particular, the electron behaves like one with a different mass.
For the physical background on this model, we refer to
the lectures by Feynman \cite{F72}.  {{Indeed, via his famous {\it path integral approach}, Feynman reduced the problem to studying the behavior of 
a three dimensional Brownian motion carrying a self-attractive interaction, which defines, in the usual Gibbs formulation, a transformed path measure weighted w.r.t. the law of Brownian paths. 
Since the nature of the self-interaction is translation-invariant in both time and space variables (see below),
no relevant information is lost by defining the same transformation  weighted w.r.t. the law of {\it Brownian increments}. This transformed path measure, or the
 {\it Polaron measure}, is the central object of interest in the present work, and our goal is to provide its {\it explicit description} in the infinite volume (i.e., large times) limit 
 and to analyze the behavior of the increments of the paths under this transformation. 
 From a physical point of view, according to Spohn (\cite{Sp87}), the long-time behavior of this path measure turns out to be crucial for a rigorous understanding of Polaron theory.
In order to put our present work into context, it therefore behooves us to allude to the quantum mechanical background of the Fr\"olich Polaron and briefly comment on 
its connections to the probabilistic questions we would like to address.}}

{{In the conventional set up, the Hamiltonian of the Fr\"ohlich Polaron is defined as the operator 
$$
H= \frac 12 p^2+ \int_{\R^3} \d k \,  a^\star(k) a(k) + \sqrt \alpha \int_{\R^3} \d k \frac{1}{|k|} \, \big[\e^{i k.x} a(k)+ \e^{-i k.x} a^\star(k)\big]
$$
which acts on a suitable Hilbert space $L^2(\R^3)\otimes \mathcal F$ with $\mathcal F$ being the Fock-space of the underlying {\it bosonic-field} interacting with the electron whose position and momentum are denoted by $x, p\in \R^3$, respectively. The bosonic field also carries the creation and annihilation operators $a^\star(k)$ and $a(k)$  which satisfy the commutation relation $[a(k),a^\star(k^\prime)]= \delta(k-k^\prime)$, while $\alpha>0$  stands for a dimensionless {\it coupling constant} which captures the strength of the interaction. Since the coupling between the electron and the bosonic-field  is translation-invariant, the total momentum $P=p+ P_f$ is conserved where $P_f=\int_{\R^3} \d k \, k \, a(k)a^\star(k)$ and 
a key object of interest is the so-called {\it energy momentum relation}, given by the bottom of the spectrum 
$$
E_\alpha(P)= \inf \,\mathrm{spec}(H_P)
$$
of the ``fiber Hamiltonian" $H_P=\frac 12 (P-P_f)^2+ \int_{\R^3} \d k \,  a^\star(k) a(k) + \sqrt \alpha \int_{\R^3} \d k \frac{1}{|k|} \, [\e^{i k.x} a(k)+ \e^{-i k.x} a^\star(k)]$.
It is known that $E(\cdot)$ is rotationally symmetric and is analytic when $P\approx 0$. Then the central objects of interest  are the {\it ground state energy}  
$$
g(\alpha)=- \min_P E_\alpha(P)
$$
 as well as the {\it effective mass} $m_{\mathrm{eff}}(\alpha)$ of the Polaron. The later quantity is defined as the inverse of the curvature:
 $$
 m_{\mathrm{eff}}(\alpha)= \bigg[\frac{\partial^2}{\partial P^2}E_\alpha(P)\big|_{P=0}\bigg]^{-1}.
 $$
 Physically relevant questions concern the {\it strong-coupling behavior} of these two objects. Indeed, the ground state energy in this regime was studied by Pekar (\cite{P49}) who also conjectured that 
that the limit $\lim_{\alpha\to\infty}\frac{g(\alpha)}{\alpha^2}=g_0>0$ exists (see below) and this was rigorously proved in \cite{DV83}. Indeed, Feynman's path-integral formulation leads to $g(\alpha)= \lim_{T\to\infty} \frac 1 T \log\langle \Psi | \e^{-TH}|\Psi\rangle$ with $\Psi$ being chosen such that its spectral resolution contains the ground state energy or low energy spectrum of $H$, but is otherwise arbitrary. Then the Feynman-Kac formula for the semigroup $\e^{-TH}$ implies that the last expression can be rewritten further as 
 \begin{equation}\label{g}
 g(\alpha)=\lim_{T\to\infty} \frac 1 T\log \E_0\bigg[\exp\bigg\{\alpha \int_0^T\int_0^T \d s \d t\,\, \frac{\e^{-{|t-s|}}}{|\omega(t)-\omega(s)|}\bigg\}\bigg]
 \end{equation}
 with $\E_0$ denoting expectation w.r.t. the law of a three-dimensional Brownian path starting at $0$. Starting with this expression, the authors in \cite{DV83-4} developed ``level-3" large deviation theory and  
 proved Pekar's conjecture in \cite{DV83} (see also Lieb and Thomas \cite{LT97} for quantitative bounds using functional analytic methods). However, the questions pertaining to the effective mass $m_{\mathrm{eff}}(\alpha)$ turned out to be much more difficult. Indeed it was Spohn (\cite{Sp87}) who, again using the path integral formulation, linked  the effective mass 
to the actual ``path behavior of the Polaron measure" -- a quantity much more subtle than its total mass. Note that in the usual Gibbs formulation, the exponential weight on the r.h.s. in \eqref{g} defines 
 a tilted measure on the path space of the Brownian motion, while its expectation provides its total mass. In 1987  Spohn (\cite{Sp87}) 
 conjectured that for any fixed coupling $\alpha>0$ and as $T\to\infty$,
 the distribution of the diffusively rescaled Brownian path under this Gibbs measure must be asymptotically Normal with zero mean and variance $\sigma^2(\alpha)>0$. Conditional on the validity of this conjecture, Spohn (\cite{Sp87}) then provided the relation
 $$
 m_{\mathrm{eff}}(\alpha)^{-1} =\sigma^2(\alpha).
 $$
 In this context, the goal of the present article is to prove Spohn's conjecture on the diffusive behavior of the Polaron measure. Actually, we will prove a stronger result
 that provides (in the limit $T\to\infty$) an explicit description of the Polaron measure itself as a mixture of Gaussian measures for any $\alpha>0$. As a corollary of this  result,  the aforementioned central limit theorem will also drop out, providing an explicit formula for the variance
 $\sigma^2(\alpha)$  or that of the effective mass $m_{\mathrm{eff}}(\alpha)=1/\sigma^{2}(\alpha)$. We may add that, using our central limit theorem, this relation $m_{\mathrm{eff}}(\alpha)=1/\sigma^{2}(\alpha)$ for the Fr\"ohlich polaron has been verified recently also in \cite[Theorem 3.1]{DS19}. We now turn to the mathematical layout of the Polaron measure and statements of our main results announced above. }}

{{\section{The Polaron measure and its large time asymptotic behavior.}}}

{{\subsection{The Polaron measure.} In the present context,  for any $\alpha>0$ and finite $T>0$ we define the Polaron measure $\widehat\P_{\alpha,T}$ by its density with respect to the  distribution of increments of Brownian motion,}} i.e., 
\begin{equation}\label{eq-Polaron-0}
\widehat\P_{\alpha,T}(\d \omega)= \frac 1 {Z_{\alpha,T}} \mathcal H_{\alpha,T}(\omega) \,\, \P(\d \omega),
\end{equation}
with
\begin{equation}\label{weight}
\mathcal H_{\alpha,T}(\omega)=\exp\bigg\{\frac{\alpha}2 \int_{-T}^T\int_{-T}^T  \frac{\e^{-|t-s|}}{|\omega(t)- \omega(s)|}\d t  \d s\bigg\}
\end{equation}
being the exponential weight.\footnote{The Polaron is sometimes written also in terms of a Kac interaction, where the weight in the exponential is given by
$\int_{-T}^T\int_{-T}^T  \frac{\eps\e^{-\eps|t-s|}}{|\omega(t)- \omega(s)|}\d t  \d s$. If we require $\eps=\alpha^{-2}$, 
this formulation turns out to be useful when studying the strong coupling limit of the Polaron, see Section \ref{sec-6}.}
Here $\P$  is Wiener Measure or three-dimensional Brownian motion, but it is defined only on the $\sigma$-field generated by the increments  $\omega(t)-\omega(s)$, $-\infty<s<t<\infty$. It can be restricted to any finite interval, in particular to $[-T,T]$, and restrictions to disjoint intervals being mutually independent. 
Also, $\alpha>0$ is the {\it{coupling parameter}} and $Z_{\alpha,T}=\E^{\P}[\mathcal H_{\alpha,T}]$ is the normalization constant or the {\it{partition function}}, which is  finite
for any $\alpha>0$ and $T>0$. 
{{As remarked earlier, the strong coupling behavior (i.e., $\alpha\to\infty$ after $T\to\infty$) of the logarithmic growth rate of  $Z_{\alpha,T}$ has been analyzed and Pekar's conjecture (\cite{P49}) was verified by Donsker and Varadhan (\cite{DV83}, see also \cite{LT97})) resulting in the following formula for the ground state limiting free energy: }}
\begin{equation}\label{eq-DV-0}
\begin{aligned}
\lim_{T\to\infty}\,\frac {1}{2 T}\log Z_{\alpha,T}= g(\alpha)= \sup_{\mathbb Q}\bigg[ \E^{\mathbb Q}\bigg\{\alpha\int_0^\infty \frac{\e^{-r}\,\d r}{|\omega(r)-\omega(0)|}\bigg\}- H(\mathbb Q)\bigg]\\
\end{aligned}
\end{equation}
and 
\begin{equation}\label{eq-DV}
\begin{aligned}
\lim_{\alpha\to\infty}\frac{g(\alpha)}{\alpha^2}=g_0=\sup_{\heap{\psi\in H^1(\R^3)}{\|\psi\|_2=1}} \,\bigg[\int\int_{\R^3\times \R^3} \frac{\psi^2(x)\,\psi^2(y)\,\d x \,\d y}{|x-y|} - \frac 12 \|\nabla\psi\|_2^2 \bigg].
\end{aligned}
\end{equation}
In \eqref{eq-DV-0}, the supremum is taken over all stationary processes $\mathbb Q$  taking values in $\R^3$, while $H(\mathbb Q)$ denotes the specific                                                                                                                                                                                                                                                                                                                                         entropy of the  process $\mathbb Q$ with respect to $\P$, while in the variational formula \eqref{eq-DV}, $H^1(\R^3)$ denotes the usual Sobolev space of square integrable functions with square integrable derivatives.
It is known (\cite{L76}) that the supremum appearing in  \eqref{eq-DV} is attained  at a function that is unique   modulo spatial translations. In other words, if $\mathfrak m$ denotes the  
set of maximizing densities, then
 \begin{equation}\label{eq-unique}
\mathfrak m=\{\psi_0^2 \star \delta_x\colon x\in \R^3\}
\end{equation}
for some $\psi_0\in H^1(\R^3)$ with $\|\psi_0\|_2=1$.

{{\subsection{The large-$T$ limit of the Polaron measure.}}} The limiting behavior of the actual path measures $\widehat\P_{\alpha,T}$ as $T\to\infty$ however has not been rigorously investigated.  We remark that, the interaction appearing in the expression for $\widehat\P_{\alpha,T}$ is {\it{self-attractve}}: 
The new measure favors paths that clump together on short time scales, i.e., the influential paths $\omega$ tend to make 
the distance $|\omega(t)-\omega(s)|$ smaller.  However, for any fixed coupling parameter $\alpha>0$, due to the presence of the 
damping factor $\e^{-|t-s|}$, one expects the interaction to stay localized as $T\to\infty$. Therefore, the following questions regarding the asymptotic behavior of the Polaron measure  arise  naturally and were 
posed in (\cite{Sp87}, Appendix 6): 
\begin{itemize}
\item Does the infinite volume Gibbs measure $\lim_{T\to\infty} \widehat\P_{\alpha,T}= {\widehat{\mathbb P}}_\alpha$ exist? Can we describe it explicitly? 
\item How mixing is $\widehat\P_\alpha$?
\item {Can we characterize the distribution 
\begin{equation}\label{eq-nu-hat}
\widehat{\nu}_{\alpha,T}=\widehat\P_{\alpha,T}\,\, \psi_T^{-1},\qquad\psi_T=\frac{1}{\sqrt{2T}}(\omega(T)-\omega(-T))
\end{equation}
of the rescaled increments $\psi_T$ under $\widehat\P_{\alpha,T}$ ?}
\item Does ${\widehat\nu}_{\alpha,T}$ converge, as $T\to\infty$, to a three dimensional centered Gaussian law $N(\mathbf 0, \sigma^2(\alpha)I)$ with variance $\sigma^2(\alpha) $? 
\item Is there an expression for  the variance $\sigma^2(\alpha)$? 
\end{itemize}
It is the goal of the present article to answer the above  questions. 

\bigskip
We  first show that, for any coupling parameter $\alpha>0$, the Polaron measure ${\widehat{\P}}_{\alpha,T}$ is a mixture of  Gaussian measures and can  therefore  be considered as a Gaussian process with a random covariance. The mixing measure ${\widehat \Theta}_{\alpha,T}$ depends on $T$ and $\alpha$, and  is explicit enough so that  we  can study its behavior as $T\to\infty$ for fixed $\alpha$. 
We show that for any $\alpha>0$, the mixing measure 
${\widehat \Theta}_{\alpha,T}$ has a limit  $\widehat{\Theta}_\alpha$  which can be described explicitly, and the limit possesses a regeneration property. This provides a useful and explicit description of the limiting  Polaron measure ${{ \widehat\P}}_\alpha$. The renewal  structure also implies mixing properties for ${{ \widehat\P}}_\alpha$.  Now, the rescaled distribution ${\widehat\nu}_{\alpha,T}$ defined in \eqref{eq-nu-hat}, which is also a mixture of spherically symmetric Gaussians,  is a Normal  distribution  with  covariance  $Z I$, where $Z \in [0,1]$ is random, and its distribution depends on $\alpha>0$ and $T>0$.  It turns out that as $T\to\infty$,  by the ergodic theorem implied by the renewal structure, one can show that  the distribution of $Z$  under ${\widehat\nu}_{\alpha,T}$ concentrates at  an explicit  constant $\sigma^2(\alpha)$ establishing the central limit  theorem for $\frac{1}{\sqrt 2T}(\omega(T)-\omega(-T))$. %We also mention that extending the method of the current article, functional CLTs for a class of translation-invariant interactions have been obtained in \cite{BP21}. 

{{\subsection{Existing methods for analyzing one-dimensional Gibbs systems.}\label{sec-compare}
 Let us now underline the crucial difficulties one faces while analyzing 
the Polaron measure using existing methods. Here we have a one-dimensional system with the function $\mathcal H$ defined as a double integral carrying the interaction 
$c(t) V(x)$ which has long range time dependence $c(t)=\e^{-|t|}$
 with an additional singularity of $V(x)=\frac{1}{|x|}$ coming from the Coulomb force.
In general, when the coupling parameter is sufficiently small and the interaction potential is smooth and bounded, 
Gibbs states corresponding to one  dimensional systems are handled by proving uniqueness
of infinite volume Gibbs measures via the well-known Dobrushin method (\cite{D68},\cite{D70}). Then exploiting the mixing properties of the limiting Gibbs measure one proves the desired central limit theorem 
(see \cite{G06} which uses this method when the interaction $W$ is sufficiently smooth and bounded and when $\alpha>0$ is small enough).
However, the method (\cite{D68},\cite{D70}) relies strongly on such requirements, in particular it fails for  interactions that are unbounded and carry singularities like $V(x)=\frac 1 {|x|}$. We also refer to another result 
of interest (\cite{BS05}) for a model coming from an Ornstein-Uhlenbeck type interaction where the proof relies upon ``linearizing"
the interaction and invoking the techniques from \cite{KV86}. This linearization technique however depends crucially on the particular  type of interaction and excludes the singular Coulomb potential.

{{Alternatively, when the time correlation function $c(t-s)$ decays slowly and $V$ is bounded or when $c(\cdot)$ has compact support and $V(x)=1/|x|$, 
one can invoke a ``Markovianization technique" which was used in \cite{M17} in a different context, see Remark \ref{remark-Nelson}. Indeed, first assume that $c(\cdot)$ has compact support so that we can split the time interval $[-T,T]$ into $O(T)$ many subintervals $I_j$ of
constant length and in the double integral in $\mathcal H_{\alpha,T}$ only interactions between ``neighboring intervals" $I_j$ and $I_{j+1}$ survive, while 
the diagonal interactions (i.e., interactions coming from the same interval $I_j$) are absorbed in the product measure $\P$  corresponding to Brownian increments on disjoint intervals. 
Then we are led to the study of a ``tilted" Markov chain on the space of increments, and it turns out that, even if the underlying interaction potential $V(x)=1/|x|$ in $\R^3$ is chosen to be singular, 
the transformed Markov chain satisfies spectral gap estimates, which then lead to fast convergence of the transformed Markov chain to equilibrium resulting in the central limit theorem for any $\alpha>0$, see \cite{M17} for details. However, when the time correlation function $c(t-s)$ decays slowly, or already when it does not have compact support (i.e., interactions like $c(t-s)=\e^{-|t-s|}$), this technique works only for interactions $V$ that are bounded. A modification of the argument requires splitting the interval $[-T,T]$ into subintervals of length $L=L(T)$ with $(T/L)^2 c(L)\to 0$ as $T\to\infty$,
while the requisite spectral theoretic estimates for the tilted Markov chain now need to hold {\it{uniformly in $T$}} which works only if $V$ is bounded and fails for the singular case $V(x)=1/|x|$. }}}}

Therefore we are led to a new approach that explicitly describes the limiting Polaron measure and in the process, also proves the central limit theorem with an explicit formula for the variance.
We will now turn to a brief description of this approach. 

{{\subsection{An outline  of the present proof.}\label{sec-outline}}}
The first crucial step of our analysis is a representation of the Polaron measure $\widehat\P_{\alpha,T}$ for any $\alpha>0$ and $T>0$, 
as a mixture of Gaussian measures. 
Note that the Coulomb potential can be written as 
$$
 \frac1{|x|}=c_0\int_0^\infty \e^{-\frac 12u^2|x|^2} \, \d u
$$
where $c_0=\sqrt{\frac{2}{\pi}}$. Then with ${\widehat\P}_{\alpha,T}=\frac 1 {Z_{\alpha,T}} \mathcal H_{\alpha,T}(\omega)\d\P$ as in \eqref{eq-Polaron-0}, we can expand the exponential weight $\mathcal H_{\alpha,T}(\omega)$ into a power series and invoke the above representation of the Coulomb potential to get 
\begin{equation}\label{eq-Polaron-01}
\begin{aligned}
\mathcal H_{\alpha,T}&=\sum_{n=0}^\infty \frac{\alpha^n}{n!} \bigg[\int\int_{-T\leq s \leq t\leq T} \frac{\e^{-|t-s|}\,\d t \, \d s}{|\omega(t)-\omega(s)|}\bigg]^n \\
&= \sum_{n=0}^\infty \frac {1}{n!} \prod_{i=1}^n \bigg[\bigg(\int\int_{-T\leq s_i \leq t_i\leq T} \big(\alpha\,\e^{-(t_i-s_i)}\,\d s_i\, \d t_i\big)\bigg)\,\,  \bigg( c_0 \int_0^\infty \, \d u_i \e^{-\frac 12 u_i^2 |\omega(t_i)-\omega(s_i)|^2}\bigg)\bigg].
\end{aligned}
\end{equation}
Note that, when properly normalized, $\mathcal H_{\alpha,T}$ is a mixture of (negative) exponentials of positive definite quadratic forms. 
Also,  in the second display in \eqref{eq-Polaron-01}, we 
have a Poisson point process taking values on the space of finite intervals $[s,t]$ of $[-T,T]$  with intensity measure $\gamma_\alpha(\d s\,\d t)=\alpha \e^{-(t-s)}\d s\d t$ on $-T\le s<t\le T$. Then it turns out that, for any $\alpha>0$ and $T>0$, 
we have a representation 
\begin{equation}\label{eq-Polaron-03}
\widehat\P_{\alpha,T}(\cdot)=\int_{\widehat\cY} \mathbf P_{\hat\xi,\hat u} (\cdot) \,{\widehat  \Theta}_{\alpha,T}(\d\hat\xi\,\d \hat u).
\end{equation}
of the Polaron measure as a superposition of Gaussian measures $\mathbf P_{\hat\xi,\hat u}$ indexed by $(\hat\xi,\hat u)\in \widehat\cY$ with $\widehat{\mathcal Y}$ being the space
of all collections of (possibly overlapping) intervals $\widehat\xi=\{[s_1,t_1],\dots,[s_n,t_n]\}_{n\geq 0}$ and strings $\widehat u\in (0,\infty)^n$, while
the ``mixing measure" ${\widehat \Theta}_{\alpha,T}$ is a suitably defined probability measure on the space $\widehat\cY$.
The details of this Gaussian representation can be found in Theorem \ref{thm1} in Section \ref{sec-2}. As an immediate corollary, we obtain 
that for any fixed $\alpha>0$ and $T>0$, the variance of any linear functional on the space of increments with respect to $\widehat\P_{\alpha,T}$ is dominated by 
the variance of the same with respect to the restriction $\P_T$ of  $\P$ to $[-T,T]$, see Corollary \ref{cor1}.

Then the limiting behavior $\lim_{T\to\infty} \widehat\P_{\alpha,T}$ of the Polaron (and hence, the central limit theorem for the rescaled increment process) 
follows once we prove a law of large numbers for the mixing measure ${\widehat\Theta}_{\alpha,T}$. This measure is defined as a tilted probability measure w.r.t. the law of the aforementioned Poisson process
with intensity $\gamma_{\alpha,T}$. Note that, the union of any collection of intervals $\{[s_i,t_i]\}$, which is a typical 
realization of this Poisson process,
need not be connected. In fact, the union is a  disjoint union of connected intervals, with gaps in between, starting and ending with gaps $[-T,\min\{s_i\}]$ and $[\max\{t_i\},T]$.
It is useful to interpret this Poisson process as a birth-death process along with some extra information (with ``birth of a particle at time $s$ and  the 
same particle dying at time $t$")  that links each birth with the corresponding death.
The birth rate is $b_{\alpha,T}(s)=\alpha(1-\e^{-(T-s)})$ and the death rate is $d_{\alpha,T}(s)=[1-\e^{-(T-s)}]^{-1}$ 
which are computed from the intensity measure $\gamma_{\alpha,T}$. 
As $T\to\infty$, the birth and death rates converge to constant birth rate $\alpha>0$ and death rate $1$, and 
we imagine the infinite time interval $(-\infty, \infty)$ to be split into an alternating sequence 
of ``gaps" and ``clusters" of overlapping intervals. The gaps are called {\it{dormant periods}} (when no individual is alive and the population size is zero) and will be denoted by $\xi^\prime$, while 
each {\it{cluster}} or an {\it{active period}} is a collection $\xi=\{[s_i,t_i]\}_{i=1}^{n(\xi)}$ of overlapping intervals with the union $\mathcal J(\xi)=\cup_{i=1}^{n(\xi)} [s_i,t_i]$ being a connected interval without any gap. Note that, inception times of both dormant and active periods possess the {\it{regeneration property}}, i.e., all prior information is lost and there is a fresh start. Also, on any dormant period $\xi^\prime$, the aforementioned Gaussian measure $\mathbf P_{\xi^\prime,u}\equiv \P$ corresponds only to the law of Brownian increments, 
and independence of increments on disjoint intervals (i.e., alternating sequence of dormant and active periods) leads to a
``product structure"  for  the mixing measure ${\widehat\Theta}_{\alpha,T}$. Indeed, if $\Pi_\alpha$ denotes the law of the above birth death process in a single active period 
with constant birth rate $\alpha>0$ and death rate $1$,
then a crucial result which is proved in Theorem \ref{thm2.6},  shows that for any $\alpha>0$, and with $\lambda(\alpha)=g(\alpha)-\alpha$ (and with $g(\alpha)$ defined in \eqref{eq-DV-0}) %there exists  
%$\lambda_0(\alpha)>0$ such that, for $\lambda>\lambda_0(\alpha)$
{{\begin{equation}\label{q}
\begin{aligned}
&q(\alpha):= \E^{\Pi_{\alpha}\otimes\mu_\alpha} \bigg[\exp\{-\lambda(\alpha)[|\mathcal J(\xi)+|\xi^\prime|]\}\mathbf F(\xi)\bigg]=1,\\
& L(\alpha):= \E^{\Pi_{\alpha}\otimes\mu_\alpha} \bigg[ \big(|\mathcal J(\xi)+|\xi^\prime|\big)\exp\{-\lambda(\alpha) [|\mathcal J(\xi)+|\xi^\prime|]\}\mathbf F(\xi)\big]<\infty.
\end{aligned}
\end{equation}}}
where $\mu_\alpha$ is exponential distribution of parameter $\alpha$ and 
$$
\begin{aligned}
&\mathbf F(\xi)=\bigg(\sqrt{\frac 2\pi}\bigg)^{n(\xi)} \int_{(0,\infty)^{n(\xi)}}\Phi(\xi,\bar u)\,\,\d \bar u \end{aligned}
$$
and $\Phi(\xi,\bar u)= \E^\P[\exp\{-\frac 12 \sum_{i=1}^{n(\xi)} u_i^2 |\omega(t_i)-\omega(s_i)|^2\}]$ is the normalizing constant for the Gaussian measure $\mathbf P_{\xi,\bar u}$ in one active period $(\xi,\bar u)$.
%It turns out that, there exists $\alpha_0,\alpha_1\in(0,\infty)$ such that when $\alpha\in(0,\alpha_0)$ or $\alpha\in(\alpha_1,\infty)$, 
%there exists $\lambda=\lambda(\alpha)$ such that $q(\lambda)=1$. 

The underlying renewal structure of the active and dormant periods imply that the mixing measure $\widehat\Theta_{\alpha,T}$ of the Polaron 
$\widehat\P_{\alpha,T}$ converges as $T\to\infty$ to the 
stationary version $\widehat{\mathbb Q}_\alpha$ on $\R$ obtained by alternating the limiting mixing measure on each active period $\xi$ defined as
$$
\begin{aligned}
&\widehat{\Pi}_\alpha(\d\xi\,\d \bar u)= \bigg(\frac {\alpha}{\lambda+\alpha}\bigg)\bigg[  \e^{-\lambda |{\mathcal J}(\xi)|}\,\,\bigg(\frac 2\pi\bigg)^{\frac {n(\xi)} 2}\,\big[\Phi(\xi, \bar u)\,\, \d \bar u\big]\bigg] \,\,\Pi_\alpha(\d\xi), \quad\mbox{where }\\
&\lambda=\lambda(\alpha)=g(\alpha)-\alpha.
\end{aligned}
$$ 
and  as the tilted exponential distribution 
$$
\widehat\mu_\alpha(\d\xi^\prime)=\bigg(\frac{\alpha+\lambda}\alpha\bigg)\,\, \e^{-\lambda|\xi^\prime|} \,\,\mu_\alpha(\d\xi^\prime) = \frac{g(\alpha)}{\alpha} \e^{-\lambda|\xi^\prime|} \,\,\mu_\alpha(\d\xi^\prime)$$
on each dormant period $\xi^\prime$ with expected waiting time  $(\lambda+\alpha)^{-1}$. 

Thus, given the Gaussian representation \eqref{eq-Polaron-03}, 
 the Polaron measure $\widehat\P_{\alpha,T}$ then converges  as $T\to\infty$, in total variation on finite intervals in $(-\infty,\infty)$,  to
$$
\widehat\P_\alpha(\cdot)=\int \mathbf P_{\hat\xi, \hat u}(\cdot)\, \,\widehat {\mathbb Q}_{\alpha}(\d\hat\xi\,\d\hat u),$$
where on the right hand side, $\mathbf P_{\hat\xi, \hat u}$ is the product of the Gaussian measures $\mathbf P_{\xi, \bar u}$ on the active intervals and law $\P$ of Brownian increments on dormant intervals, and 
the integral above is taken over the space of all active intervals (with $\bar u=(u_i)_{i=1}^{n(\xi)}$ and $u_i$'s being attached to each birth with the corresponding death) as well as dormant intervals. The central limit theorem
$$
\begin{aligned}
\lim_{T\to\infty} \widehat\P_{\alpha,T}\bigg[\frac {\omega(T)-\omega(-T)}{\sqrt{2T}} \in \cdot\bigg]&= \lim_{T\to\infty} \widehat\P_{\alpha}\bigg[\frac {\omega(T)-\omega(-T)}{\sqrt{2T}} \in \cdot\bigg] \\
&= \mathbf N (0, \sigma^2(\alpha) \mathbf I), \qquad\mbox{with  }\sigma^2(\alpha)\in (0,1),
\end{aligned}$$
for the rescaled increment process $(2T)^{-1 /2}\, [\omega(T)-\omega(-T)]$ under the finite-volume limit $\widehat\P_{\alpha,T}$ as well as that under its infinite-volume counterpart
$\widehat\P_\alpha$ as $T\to\infty$ with the same variance $\sigma^2(\alpha)$ also follow
 readily. It turns out that the variance in each dormant period $\xi^\prime$ is just the expected length $(\alpha+\lambda)^{-1}$ of the empty period, 
 and the resulting central limit 
 covariance matrix is $\sigma^2(\alpha) I$, where for any unit vector $v\in\R^3$ and any active period $\xi=[0,\sigma^\star]$, 
 $$
 \begin{aligned}
  \sigma^2(\alpha)&= \lim_{T\to\infty} \frac 1 {2T}\E^{\widehat{\P}_{\alpha,T}}\bigg[\big\langle v, \omega(T)-\omega(-T)\big\rangle^2\bigg]\\
&=   \frac{(\alpha+\lambda)^{-1}+ \E^{\widehat{\Pi}_\alpha}\big[\E^{\mathbf P_{\xi,\bar u}}[\langle v,\omega(\sigma^\star)-\omega(0)\rangle^2]\big]}{(\alpha+\lambda)^{-1}+  \E^{\widehat{\Pi}_\alpha}[\sigma^\star]}\\
&=  \frac{g(\alpha)^{-1}+ \E^{\widehat{\Pi}_\alpha}\big[\E^{\mathbf P_{\xi,\bar u}}[\langle v,\omega(\sigma^\star)-\omega(0)\rangle^2]\big]}{g(\alpha)^{-1}+  \E^{\widehat{\Pi}_\alpha}[\sigma^\star]}\in {{(0,1)}}.
 \end{aligned}
 $$
 The proofs of the limiting assertions $\lim_{T\to\infty} \widehat\P_{\alpha,T}$ and the central limit theorem are carried out in Section \ref{sec-4}. 
 
 {{We end this discussion with some relevant remarks.}}

{{\begin{remark}[Lower bound on the effective mass]\label{remark-mass}
 As remarked earlier, in \cite{Sp87} assuming that a CLT for the Polaron measure holds, the relation $m_{\mathrm{eff}}(\alpha)^{-1}=\sigma^2(\alpha)$ between the effective mass and the CLT variance $\sigma^2(\alpha)$ was provided (this relation has been recently 
 rigorously shown also in \cite[Theorem 3.1]{DS19} using the CLT proved in the current article for the Fr\"ohlich polaron). The attractive nature of the interaction in the Polaron measure is reflected in our estimate $\sigma^2(\alpha)\in (0,1)$ implying the strict bound $m_{\mathrm{eff}}(\alpha)\in (1,\infty)$ and underlining the increment of the mass of electron coupled with the bosonic field.
\end{remark}

%\begin{remark}[Intermediate coupling]\label{remark-alpha}
%Note that our results hold true for both weak and strong coupling regime, i.e. for $\alpha\in (0,\alpha_0)\cup (\alpha_1,\infty)$ with a potential gap at an intermediate coupling regime 
%$[\alpha_0,\alpha_1]$. However, the same statements should hold for any coupling $\alpha>0$. This potential gap originates from the solvability for any $\alpha>0$ 
%of the equation $q(\lambda)=1$ (recall \eqref{q}) for some  $\lambda=\lambda(\alpha)$. However, since
% the physically prominent cases are covered by our main results, we presently content ourselves 
% with the cases pertaining to both weak and strong coupling regimes (see also Remark \ref{remark-sec6} below).
% \end{remark}

\begin{remark}[The Polaron measure in strong coupling $\alpha\to\infty$ and the Pekar process]\label{remark-sec6}
In Section \ref{sec-6} we conclude with a discussion on the {\it strong coupling limit} of the limiting Polaron measure $\widehat{\mathbb P}_\alpha$ as $\alpha\to\infty$ and its connection with the increments of a stationary stochastic process, or the increments of the so-called {\it Pekar process}, which is determined uniquely by {\it any} solution $\psi$ of the Pekar variational formula $g_0$ defined in \eqref{eq-DV}. The detailed proofs can be found in our recent work \cite{MV18}.
 \end{remark}

\begin{remark}[Related models in quantum mechanics]\label{remark-Nelson}
The Fr\"ohlich Polaron considered in the present paper belongs to a large class of quantum mechanical models which capture the case of an electron interacting with a scalar bosonic field studied by Nelson (\cite{N64}) in the context of energy renormalization. To complete the picture we briefly comment on the state of the art of the available rigorous results pertaining to these models.  Mathematically, the scalar bosonic-field translates to an infinite-dimensional Ornstein-Uhlenbeck (OU) process $\{\varphi(x,t)\}_{x\in \R^d, t>0}$ with covariance structure 
\begin{equation}\label{eq0:Nelson}
\begin{aligned}
\int \varphi(x,t)\, \varphi(y,s) \, \mathbf P^{\mathrm{OU}}(\d\varphi)&=  \int_{\R^3} \d k\,\, |\widehat\rho(k)|^2\, \frac 1 {2\omega(k)} \, \e^{-\omega(k) |t-s|} \,\, \e^{\mathrm i k\cdot (x-y)} \\
&=: \mathcal W(t-s, x-y).
\end{aligned}
\end{equation}
Here $\widehat\rho$ denotes the Fourier transform of the mass distribution of the quantum particle, while $\omega$ stands for the Phonon dispersion relation. \footnote{In quantum mechanics it is customary to denote this dispersion relation by $\omega$. This is not to be confused with the sample path of the Brownian motion and to keep notation disjoint, the latter object 
is denoted by $x(\cdot)$ in Remark \ref{remark-Nelson}.}
 Now with a Hamiltonian $-\frac 12 \Delta+ \alpha \sqrt 2 \varphi(x,t)$, the Feynman-Kac formula 
leads to the path measure  
\begin{equation}\label{eq1:Nelson}
\frac 1 {Z_T} \, \exp\bigg\{- \alpha\sqrt 2\int_0^T \varphi\big(x(t),t)\big) \d t\bigg\} \mathbf P^{\mathrm{OU}}(\d\varphi)\otimes P(\d x) %\qquad e>0.
\end{equation}
where $P_0$ denotes the law of a Brownian path $x(\cdot)$. The exponent above is linear in $\varphi$ and integration w.r.t. the Gaussian measure $\mathbf P^{\mathrm{OU}}$, together with \eqref{eq1:Nelson} now leads to the Gibbs measure with an exponential weight $\exp\{\alpha\int_0^T\int_0^T \mathcal W(t-s,x(t)-x(s)) \d s \d t\}$ on the Wiener space. The case of the Fr\"ohlich Polaron corresponds to the case
$\omega\equiv 1$ and $\widehat\rho(k)=|k|^{-1}$ in \eqref{eq0:Nelson} and thus $\mathcal W(t,x)=\frac{\e^{-|t|}}{|x|}$. Another case of physical prominence is that of {\it massless Bosons} which requires the choice $\omega(k)=|k|$ and a radially symmetric $\widehat\rho$  with a fast decay at infinity with $\widehat\rho(0)=0$. This choice in \eqref{eq0:Nelson} leads to the interaction potential $\mathcal W(t,x)= \int_0^\infty \d r \, \widehat\rho(r) \, \e^{-r |t|} \frac{\sin(r|x|)}{|x|}$. Like the case of the Fr\"ohlich Polaron, a satisfactory analysis of Gibbs measures corresponding to such interactions also 
do not succumb to the aforementioned Dobrushin method. Developing 
the Markovian approach discussed in Section \ref{sec-compare}, it was shown in \cite{M17} that a CLT for the increment process holds for any coupling parameter $\alpha$ for 
long-range in time and bounded in space interactions satisfying $\sup_x |\mathcal W(t,x)| \leq \frac C{1+t^\gamma}$ for $\gamma>2$ (such interactions come naturally from the above assumptions from $\widehat\rho$, a special case of interest is $\mathcal W(t,x)=1/(1+|x|^2+t^\gamma)$ for $\gamma>2$) or for the short-range but singular interaction of the form $\mathcal W(t,x)=c(t)V(x)$ where $c$ has compact support and $V(x)=\delta_0(x)$ in $d=1$ or $V(x)=1/|x|^{p}$ for $p\in (0,\frac 2{d-2})$ in $d\geq 3$. However, the latter method does not seem to cover interactions which are long-range in time {\it and} unbounded in space like $\mathcal W(t,x)=\e^{-|t|}/{|x|}$ corresponding to the Fr\"ohlich Polaron analyzed in the present article.
\end{remark}
}}

\begin{remark}
Extending the method developed currently, functional CLT for a class of translation-invariant interactions of the form $\mathcal W(t-s,x(t)-x(s))$ as in Remark \ref{remark-Nelson} have been obtained in \cite{BP21}. The result on the CLT there requires (apart from the assumptions needed for existence of the infinite volume limit)
an additional hypothesis about quasi-concavity of $x\mapsto \mathcal W(\cdot,x)$ to apply Gaussian correlation inequalities.\footnote{These assumptions cover Fr\"ohlich polaron (unlike \cite{M17}), but are more restrictive than \cite[Assumption A]{M17} which does not need any concavity in the spatial component $x\mapsto \mathcal W(\cdot,x)$.} While the method there follows our current approach, the argument there (unlike our proof of Theorem \ref{thm2.6} in Section \ref{sec-thm2.6})
relies additionally on using known results from quantum mechanics for the Fr\"ohlich Polaron (e.g. existence of a ground state at zero total momentum, spectral gap etc.). 
 \end{remark}

\noindent{{\it Organization of the rest of the article:}}We now briefly comment on the organization of the rest of the article. Section \ref{sec-2} is devoted to the representation of the Polaron measure as a superposition of Gaussian measures w.r.t. a mixing measure,
while Section \ref{sec-3} is devoted to the estimates with respect to the mixing measure. The identification of the limiting Polaron measure as well as the  
central limit theorem for the increment process are carried out in Section \ref{sec-4}. In Section \ref{sec-6} we conclude with a brief discussion on the strong coupling limit regime of the Polaron measure $\widehat\P_\alpha$. In an Appendix (Section \ref{sec-appendix}) we have collected some estimates w.r.t. birth and death processes which are used only in the proof of the exponential mixing property of $\widehat\P_\alpha=\lim_{T\to\infty} \widehat\P_{\alpha,T}$ in Theorem \ref{thm4.5} (but neither in the proof of Theorem \ref{thm4} nor in that of Theorem \ref{thm5}).

\section{Polaron as a superposition of Gaussian measures}\label{sec-2}

We will denote by $\Omega  = C\big((-\infty,\infty);\R^3)$ the space of continuous functions $\omega$ taking values in $\R^3$. We will work with the probability space
$(\Omega,\mathcal F,\P)$, where $\mathcal F$ is the $\sigma$-algebra generated by the increments $\{\omega(t)-\omega(s)\}$, while $\P$ is the Gaussian measure 
governing the law of three dimensional Brownian increments over intervals in $(-\infty,\infty)$. 

For convenience, we will introduce the following notation which we will use in this section and the rest of the article. We will denote by $\cX_n$ the space of collections $\widehat\xi=\{[s_1,t_1],\dots,[s_n,t_n]\}$ of $n$ (possibly overlapping) intervals. We will write 
\begin{equation}\label{eq-X}
\widehat{\mathcal X}=\bigcup_{n=0}^\infty \mathcal X_n \quad\mbox{and}\,\,\, \widehat\cY=\bigcup_{n=0}^\infty \bigg(\mathcal X_n\otimes (0,\infty)^n\bigg).
\end{equation}
Typical elements of the space $\widehat\cX$ and $\widehat\cY$ will be denoted by $\widehat\xi\in \widehat\cX$ and $(\widehat\xi,\widehat u)\in \widehat{\mathcal Y}$, respectively.

\subsection{Quadratic forms on dual spaces and Gaussian measures.}
We will consider other centered Gaussian processes which are defined on the same $\sigma$-field $\mathcal F$ generated by increments, and 
these processes will be labeled  through  their quadratic forms defined as follows. Let $\Mcal_0$ be the space of compactly supported signed measures $\mu$ on the real line $\R$ 
 with total mass $\mu({\mathbb R})=0$. Then, $\mathcal Q(\mu)=\E^{\P}\big[(\int_\R \omega(s)\mu(\d s))^2\big]$ will define the quadratic form on $\Mcal_0$ 
 for one dimensional Brownian increments $\P$. i.e., with $F(s)=\mu((-\infty, s])$,
\begin{equation}\label{eq-Q-mu}
\mathcal Q(\mu)=\int_{-\infty}^\infty |F(s)|^2 \d s=\sup_{\omega(\cdot) }\bigg[ 2\int_\R \omega(s)\mu(\d s)-\int_\R [\omega^\prime(s)]^2\,\,\d s\bigg],
\end{equation}
with the supremum above being taken over absolutely continuous functions $\omega$ with square integrable derivatives.

Let $\widehat\xi=\{[s_1, t_1], [s_2,t_2],\dots,[s_n,t_n]\}\in \widehat\cX$ be a collection of $n$ possibly overlapping  intervals in $\R$. For any such $\widehat\xi$ and vector $\widehat u=(u_1,u_2,\ldots, u_n)\in (0,\infty)^n$, we can again define a quadratic form
\begin{equation}\label{eq-Q-xi-u}
\mathcal Q_{\hat\xi,\hat u} (\mu)=\sup_{\omega(\cdot) }\bigg[ 2\int \omega(s)\mu(\d s)-\int [\omega^\prime(s)]^2\,\,\d s-\sum_{i=1}^n u_i^2 [\omega(t_i)-\omega(s_i)]^2\bigg],
\end{equation}
 
Then the  corresponding Gaussian measure  will be denoted by $P_{\hat \xi,\hat u}$, i.e.,  
\begin{equation}\label{eq-P-xi-u}
\E^{P_{\hat \xi,\hat u}}\bigg[\bigg(\int_\R \omega(s)\mu(\d s)\bigg)^2\bigg]= \mathcal Q_{\hat \xi, \hat u} (\mu),
\end{equation}
and we can take three independent copies of $P_{\hat\xi,\hat u}$ to get a three dimensional version and we will  denote it $\mathbf P_{\hat\xi,\hat u}$. We then have a collection $\{\mathbf P_{\hat\xi,\hat u}\}_{(\hat\xi,\hat u)\in\widehat{\mathcal Y}}$ of Gaussian processes indexed by $(\widehat\xi,\widehat u)\in \widehat{\mathcal Y}$.
Throughout the rest of the article,
we will also denote by 
\begin{equation}\label{eq-F-xi-u}
\Phi(\hat\xi,\hat u)=\E^\P\bigg[\exp\bigg\{-\frac 12\sum_{i=1}^n u_i^2 |\omega(t_i)-\omega(s_i)|^2\bigg\}\bigg].
\end{equation}
the normalizing constant for the Gaussian measure $\mathbf P_{\hat\xi, \hat u}$.

We now take note of the following fact. Suppose we have collections $\{\xi_r\}_r$ with each
$$
\xi_r=\bigg\{ [s_i,t_i] \bigg\}_{i=1}^{n(r)}
$$
being a collection of $n(r)$ overlapping sub-intervals $[s_i,t_i]$, such that their {{unions}}
 $$
\mathcal J_r=\mathcal J(\xi_r)= \bigcup_{i=1}^{n(r)} [s_i,t_i]
 $$
{ which are again intervals, and are mutually disjoint} (i.e., $\mathcal J_r\cap \mathcal J_{r^\prime}=\emptyset$ if $r\neq r^\prime$). 
Then if  $\mu_r \in \Mcal_0$ with $\supp(\mu_r)\subset \mathcal J_r$, then for any $\bar u_r:=(u_1,\dots, u_{n(r)})\in (0,\infty)^{n(r)}$
\begin{equation}\label{eq-P-xi-u-independence}
\mathcal Q_{\hat\xi, \hat u}\bigg(\sum_r \mu_r\bigg)=\sum_r \mathcal Q_{\xi_r, \bar u_r}(\mu_r),
\end{equation}
where $\hat\xi=\cup_r \xi_r$ and $\hat u=\{\bar u_r\}_r$, with each quadratic form $\mathcal Q_{\xi_r, \bar u_r}(\mu_r)$ being defined as in \eqref{eq-Q-xi-u}, i.e., 
\begin{equation}\label{eq-Q-xi-u-2}
\mathcal Q_{\xi_r, \bar u_r}(\mu_r)= \sup_{\omega(\cdot) }\bigg[ 2\int \omega(s)\mu_r(\d s)-\int [\omega^\prime(s)]^2\,\,\d s-\sum_{i=1}^{n(r)} u_i^2 [\omega(t_i)-\omega(s_i)]^2\bigg].
\end{equation}
The corresponding Gaussian measure will be denoted by $\mathbf P_{\xi_r, \bar u_r}$. This proves the mutual independence of the restrictions of the previously defined Gaussian measure $\mathbf P_{\hat\xi,\hat u}$ to the disjoint collections $\xi_r$.

\subsection{The Poisson point process and the mixing measure.}

Let us fix a finite $T>0$ and $\alpha>0$. Let $\Gamma_{\alpha,T}$ denote the law of the Poisson point process with intensity measure 
$$
\gamma_{\alpha,T}=\alpha \e^{-(t-s)} \,\ 1_{-T\leq s<t \leq T} \,\, \d s \, \d t.
$$
Then $\Gamma_{\alpha,T}$ is a probability measure on the space $\widehat\cX$ and each realization of the point process is given by a random number $n$ of possibly overlapping 
intervals $\{[s_i,t_i]\}_{i=1}^n$. As remarked earlier, the union of these intervals need not be connected, 
and will be a union of disjoint intervals, with gaps in-between and each interval being a union of overlapping sub-intervals $\{[s_i,t_i]\}_{i=1}^{n(r)}$, with $n=\sum_r n(r)$. 
We will call each $\xi_r=\{[s_i,t_i]\}_{i=1}^{n(r)}$ an {\it active period}, or a {\it{cluster}}, and these clusters will be separated by gaps that we will call {\it dormant} periods and denote them by $\xi^\prime_r$.
Dormant and active periods alternate, beginning and ending with dormant periods $\xi^\prime_1=[-T,\min_i s_i]$ and $\xi^\prime_k=[\max_i t_i,T]$, splitting the interval $[-T,T]$  into a collection 
$$
\widehat\xi=\bigg\{\xi^\prime_1,\xi_1,\dots,\xi^\prime_{k-1},\xi_{k}, \xi_{k+1}^\prime\bigg\}
$$
of $k+1$ dormant intervals $\{\xi^\prime_{r}\}$ and $k$ active intervals
$\{\xi_{r}\}$. Then, with $u_r=(u_1,\dots, u_{n(r)})\in (0,\infty)^{n(r)}$,  the quadratic form $\mathcal Q_{\xi_r,u_r}$ defined in \eqref{eq-Q-xi-u} also provides
a Gaussian measure $\mathbf P_{\xi_{r},u_r}$ on each active period $\xi_{r}$, while on any of the dormant interval $\xi^\prime_r$, this Gaussian measure coincides with the laws of Brownian increments $\P$,
which is of course given by the quadratic form $\mathcal Q$ (recall \eqref{eq-Q-mu}). Thanks to independent increments on disjoint periods, the normalization constant defined in \eqref{eq-F-xi-u} also
splits as the product
$$
\Phi(\hat\xi,\hat u)=\prod_{r=1}^k \Phi(\xi_{r},\bar u_{r})= \prod_{r=1}^k \E^\P\bigg[\exp\bigg\{-\frac 12\sum_{i=1}^{n(r)} u_i^2 |\omega(t_i)-\omega(s_i)|^2\bigg\}\bigg],
$$
which combined with the earlier remark (recall \eqref{eq-P-xi-u-independence}, leads to the factorization 
\begin{equation}\label{eq-P-xi-u-2}
\begin{aligned}
\mathbf P_{\hat \xi, \hat u}=\prod_{r=1}^{k} \mathbf P_{\xi_r, \bar u_r} \qquad &(\widehat\xi,\widehat u)\in \widehat{\mathcal Y}\,\,\,\mbox{with}\,\,\,\widehat\xi=\big\{\xi^\prime_1,\xi_1,\dots,\xi^\prime_{k-1},\xi_{k}, \xi_{k+1}^\prime\big\} \quad\mbox{and} \\
&\qquad \widehat u=(\bar u_{r})_{r=1}^k, \,\,\bar u_{r}\in (0,\infty)^{n(r)},
\end{aligned}
\end{equation}
of the Gaussian measure on increments that is independent over different $\xi_r$. 

Then with $\Gamma_{\alpha, T}$ being the law of the point process with intensity $\gamma_{\alpha,T}$ (i.e., $\Gamma_{\alpha,T}$ is a probability measure on the space $\widehat\cX$,  recall \eqref{eq-X}), 
for any $\lambda$, since $\sum_{r=1}^k |{\mathcal J }(\xi_r)|+\sum_{r=1}^{k+1}|\xi_r'|=2T$, we can write our mixing measure 
${\widehat\Theta}_{\alpha,T}$ on the space $\widehat{\mathcal Y}$ as 
\begin{equation}\label{eq-Theta-hat-T}
\begin{aligned}
\widehat\Theta_{\alpha,T}\big(\d \hat\xi\,\d \hat u\big)&= \frac {\e^{2\lambda T}} {Z_{\alpha,T}} \,\, \bigg[\bigg(\sqrt{\frac 2\pi}\bigg)^{n} \,\,\Phi(\hat\xi,\hat u) \,\,\d\hat u\bigg]\,\e^{-\lambda(\sum_{r=1}^k |{\mathcal J}(\xi_r)|+\sum_{r=1}^{k+1}|\xi^\prime_r|)}    \,\Gamma_{\alpha,T}(\d\hat\xi)
\\
&=\frac {\e^{2\lambda T}} {Z_{\alpha,T}} \,\, \prod_{r=1}^k\bigg[\bigg(\sqrt{\frac 2\pi}\bigg)^{n(r)} \,  \e^{-\lambda |{\mathcal J}(\xi_r)|}\, \Phi(\xi_{r},\bar u_{r}) \,\,\d \bar u_r\bigg] \,  e^{-\lambda\sum_{r=1}^{k+1} |\xi^\prime_r|} \,\Gamma_{\alpha, T}(\d\hat\xi)
\end{aligned}
\end{equation}
where $Z_{\alpha,T}$ is the normalizing constant of the Polaron measure that also makes 
$\widehat\Theta_{\alpha,T}$ a probability measure on  $\widehat{\mathcal Y}$.

We are now ready to state the main result of this section. Recall that if $\P_T$ denotes the restriction of $\P$ to the finite time interval $[-T,T]$, then 
$$
\widehat\P_{\alpha,T}(\d\omega)= \frac 1 {Z_{\alpha,T}}\,\,\mathcal H_{\alpha,T}(\omega) \,\,\P_T(\d\omega)
$$
defines the finite volume Polaron measure with exponential weight
$$
\mathcal H_{\alpha,T}(\omega)= \exp\bigg[\frac\alpha2\int_{-T}^T\int_{-T}^T \,\,\d t\d s\,\, \frac{\e^{-|s-t|}} {|\omega(t)-\omega(s)|}\bigg], 
$$
and normalizing constant $Z_{\alpha,T}$.

Here is the statement of our first main result. 
\begin{theorem}\label{thm1}
Fix any  $\alpha>0$ and  $T>0$.  Then there exists a probability measure ${\widehat\Theta}_{\alpha,T}$ on the space $\widehat{\mathcal Y}$ defined in \eqref{eq-Theta-hat-T}, such that  \begin{equation}\label{eq-Gaussian-mixture}
\widehat\P_{\alpha,T}(\cdot)=\int_{\widehat{\mathcal Y}} \, \mathbf P_{\hat\xi,\hat u}(\cdot)\,\,\widehat\Theta_{\alpha,T}(\d\hat\xi\, \d \hat u),
\end{equation}
where, for any $(\hat\xi,\hat u)\in \widehat{\mathcal Y}$, $\mathbf P_{\hat \xi,\hat u}$ is the centered Gaussian measure on increments defined in \eqref{eq-P-xi-u-2}.
\end{theorem}
\begin{proof}
Let us recall \eqref{eq-Polaron-01} from Section \ref{sec-outline}. Then we have
$$
\mathcal H_{\alpha,T}(\omega)= \sum_{n=0}^\infty \frac {1}{n!} \prod_{i=1}^n \bigg[\int\gamma_{\alpha,T}(\d s_i\,\d t_i)\int_0^\infty\bigg(\sqrt{\frac 2\pi}\,\d u_i\bigg)\bigg(  \e^{-\frac 12 u_i^2 |\omega(t_i)-\omega(s_i)|^2}\bigg)\bigg].
$$
Then the { definitions of the Gaussian measure $\mathbf P_{\hat\xi,\hat u}$ and that of the mixing measure $\widehat\Theta_{\alpha,T}(\d\hat\xi,\d\hat u)$ complete} the proof of Theorem \ref{thm1}
\end{proof}
The following corollary asserts that the variance under $\P_T$  dominates the variance under Polaron $\widehat\P_{\alpha,T}$ for any fixed $\alpha>0$ and $T>0$.
\begin{corollary}\label{cor1}
 
For any $\alpha>0$ and $T>0$, and for any unit vector $v\in\R^3$, 
$$
\E^{\widehat\P_{\alpha,T}}\big[\big\langle v, \omega(T)-\omega(-T)\big\rangle^2\big] \leq \E^{\mathbb P_T}\big[\big\langle v, \omega(T)-\omega(-T)\big\rangle^2\big]
$$
\end{corollary}
\begin{proof}
Since for any $\mu\in \Mcal_0$ and $(\hat\xi,\hat u)\in \widehat{\mathcal Y}$, comparing \eqref{eq-Q-mu} and \eqref{eq-Q-xi-u}, we have
$$
\mathcal Q_{\hat\xi,\hat u} (\mu)\leq \mathcal Q(\mu),
$$
the proof of the claimed monotonicity is obvious. 
\end{proof}

\section{Some estimates with respect to birth and death processes.}\label{sec-3}

\subsection{A birth-death process $\Pi_\alpha$ on a single active period.}\label{sec-Q-alpha}
Let $\Pi_\alpha$ denotes the law of a birth-death process starting with population size $1$ at time $0$, and birth rate $\alpha>0$ and death rate $1$. It 
 is described by a continuous time Markov chain $(N_t)_{t\geq 0}$ taking values in $\Z_+=\{0,1,2,\dots\}$ with jump rates 
 \begin{equation}\label{eq-jump-rate}
 a_{n,n+1}= \alpha\qquad a_{n,n-1}=n
 \end{equation}
 if $n\geq 1$ denotes the curent population size. Then the waiting time at state $n$ until the next event of birth or death is exponentially distributed with parameter $n+\alpha$, and the probabilities 
 of jumping to $n+1$ and $n-1$ are respectively $\alpha/(n+\alpha)$ and $n/(n+\alpha)$. We will also denote the successive jump times of this continuous time Markov chain as $\{\sigma_j\}$.
 Note that the evolution of this birth-death process then describes an active period, which starts at the birth of an individual and lasts until the last death, i.e., at time 
 \begin{equation}\label{eq-extinction}
 \sigma^\ast=\inf \{t>0\colon N(t)=0\}.
 \end{equation}
 Note that we also have an embedded discrete time Markov chain 
$$
 X_j=N(\sigma_j+0) \qquad X_0=1.
$$
with transition probabilities 
 \begin{equation}\label{eq-Markov-chain}
 \mathrm{Prob}\big\{X_{j+1}=n+1| X_j=n\big\}=\frac{\alpha}{n+\alpha}\qquad  \mathrm{Prob}\big\{X_{j+1}=n-1| X_j=n\big\}=\frac{n}{n+\alpha}.
 \end{equation}
 Note that this population size Markov chain will hit $0$ after $\ell$ steps where $\ell=\inf \{j: X_j=0\}$ and  $\sigma_\ell=\sigma^\ast$ is the extinction time, and $\ell=2n-1$ if $n-1$  is the number of new births. 
 
 Furthermore, we have the lifetimes of the individuals $\{[s_i,t_i]\}$, 
 and we will write
 $$
\mathcal J(\xi)=\big[\min_i s_i,\,\,\max_i t_i\big] =[0,\sigma^\star]
 $$
for the time-span of the active period $\xi=\{[s_i,t_i]\}_{i=1}^{n(\xi)}$ with $n(\xi)\geq 1$ individuals. We assume without loss of generality that
$t_1<t_2\cdots<t_{n}=\sigma^\ast$. Each $t_i=\sigma_r$ for some $r=r_i$. We denote by 
\begin{equation}\label{eq-delta-def}
\delta_i=\sigma_{r_i}-\sigma_{r_i-1},
\end{equation}
and note that given $X_{r_i-1}$, the distribution of $\delta_i$ is exponential with rate $\alpha+X_{r_i-1}$.  The life times 
\begin{equation}\label{eq-tau-def}
\tau_i=t_i-s_i
\end{equation}
are all exponentials with rate $1$.

For any single active period $\xi$, throughout the rest of the article, we will write 
\begin{equation}\label{eq-Phi-1}
\mathbf F(\xi)= \bigg(\sqrt{\frac 2\pi}\bigg)^{n(\xi)} \int_{(0,\infty)^{n(\xi)}}\,\d \bar u\,\, \Phi(\xi,\bar u)
\end{equation}
with $n(\xi)\in \N$ being the number of  individuals that constitute $\xi$, and as usual, 
$$
\Phi(\xi,\bar u)=\E^\P[\exp\{-\frac 12\sum_{i=1}^{n(\xi)} u_i^2 |\omega(t_i)-\omega(s_i)|^2\}].
$$ 
Also $\mu_\alpha$ will denote the exponential distribution with parameter $\alpha$, and $\Pi_\alpha$ is the law of a single active period $\xi$, i.e., the law of the birth-death process with birth rate $\alpha>0$ and death rate $1$ starting with one individual at time $0$, along with  the information that matches the birth and death of each individual.

\begin{remark}
For our purposes it is convenient to use some further notation. Recall that  in an active period $\xi=\{[s_i,t_i]\}_{i=1}^{n(\xi)}$ we take the starting time as $0$ with population size as $1$ and then we can have  certain number $n(\xi)-1$
additional births before the population becomes extinct and we have the lifetimes $\{[s_i,t_i]\}$ of these individuals 
$t_i>s_i$ for $i=1,2,\cdots,n(\xi)$. In addition, the union $\mathcal J(\xi)=\cup_{i=1}^{n(\xi)} [s_i,t_i]$ is again an interval without gaps, denoting the time span of the active period $\xi$. 
We can also think of this time span as the history or the excursion of a single active period $\xi$. For notational convenience, we will write
\begin{equation}\label{eq-Y}
\mathcal X=\{\xi\} \qquad \mathcal Y=\{(\xi,\bar u)\}
\end{equation}
such that $\xi=\{[s_i,t_i]\}_i$ is an active period (i.e., $\xi$ is a collection of finitely many overlapping intervals whose union $\mathcal J(\xi)=\cup_{i=1}^{n(\xi)}[s_i,t_i]$ is again an interval),  and $\bar u=(u_i)_{i=1}^{n(\xi)}$ is
a positive vector with each $u_i$ attached to the information linking birth at $s_i$ and death at $t_i$. 
\end{remark}

\subsection{A birth-death process $\Pi_{\alpha,T}$ depending on terminal time $T$ on a single active period}\label{sec-Q-alpha-T}

We recall that for each fixed $\alpha>0$ and $T>0$, we have the law $\Gamma_{\alpha,T}$ of the Poisson point process with intensity measure $\gamma_{\alpha,T}(\d s \d t)= \alpha \e^{-(t-s)} \, \1_{-T\leq s < t\leq T}\, \d s \, \d t$. 
We can also have a birth-death process whose distribution is obtained from restricting this Poisson process to {\it{the excursion of the first active period}} $(\xi,u)\in \cY$ with $\mathcal J(\xi)\subset [-T,T]$.
We will  denote by $\Pi_{\alpha,T}$ the probability distribution of this birth and death process  
on the first excursion $(\xi, \bar u)\in \mathcal Y$ (starting from a population of size $1$), and both birth and death rates of this process will depend on the ``remaining time": The birth rate corresponding to this process $\Pi_{\alpha,T}$ is given by the marginal 
\begin{equation}\label{eq-birth-rate-T}
b_{\alpha,T}(s)=\int_s^T\alpha \e^{-(t-s)} \d t=\alpha(1-\e^{-(T-s)}),
\end{equation}
while the death rate is computed as
\begin{equation}\label{eq-death-rate-T}
d_{\alpha,T}(t)=-\frac{\d}{\d t}\bigg[\log\bigg(\frac{\int_t^T \e^{-(s-a)}\d s}{\int_a^T \e^{-(s-a)} \d s}\bigg)\bigg]=\frac{1}{1-\e^{-(T-t)}}.
\end{equation}

\begin{remark}\label{rmk-bd-1}
Note that  the Poisson point process whose realizations are intervals $[s,t]\subset [-T,T]$ can be recovered from the Poisson process with realizations $[s,t]\subset (-\infty,\infty)$ 
by simply deleting intervals that are not contained in $[-T,T]$. Likewise, the birth-death process $\Pi_{\alpha,T}$ can be obtained from the law of the birth death process $\Pi_\alpha$ defined in Section \ref{sec-Q-alpha} by trimming down records of individuals
whose lifespan exceeds the terminal time $T$.
\end{remark}
\begin{remark}\label{rmk-bd-2}
It follows from \eqref{eq-birth-rate-T} and \eqref{eq-death-rate-T} that both birth and death rates $b_{\alpha,T}(\d s)$ and $d_T(\d s)$ at a given time $s$ depend on the 
``leftover time" $(T-s)$. In particular, the birth-rate $b_{\alpha,T}(s)$ drops to $0$, while the death rate $d_T(s)$ shoots up like $\frac 1{T-s}$ as the terminal time $T$ is approached.
\end{remark}
\begin{remark}\label{rmk-bd-3}
Note that the birth rate $b_{\alpha,T}$ and death rate $d_{\alpha,T}$ of $\Pi_{\alpha,T}$ converges to the birth rate $\alpha$ and death rate $1$ of $\Pi_\alpha$
as long as the remaining time is large enough. Moreover for finite $T$, the birth rate is smaller and the death rate is higher, 
and as $T\to\infty$,  $\Pi_{\alpha,T}$ converges to to $\Pi_\alpha$ as $T\to\infty$ in the total variation distance in the space of probability measures on $\cX$. 
\end{remark}

\subsection{An estimate on a single renewal period}\label{sec-thm2.6}

\begin{theorem}\label{thm2.6}
Fix any $\alpha>0$ and let $\lambda(\alpha)=g(\alpha)-\alpha$ with $g(\alpha)=\lim_{T\to\infty}\frac 1 {2T}\log Z_{\alpha,T}$ as in \eqref{eq-DV-0}. Then with $\Pi_\alpha$ being the law of the birth-death process on any active period $\xi$ with birth rate $\alpha$ and death rate $1$ (starting with population size $1$ at time zero), and $\mu_\alpha$ being the exponential distribution with parameter $\alpha$ on a dormant period $\xi^\prime$, 
we have 
\begin{align}
&q(\alpha):=\E^{\mu_\alpha\otimes\Pi_\alpha}\big[\e^{-\lambda(\alpha)(|\xi^\prime |+|\mathcal J(\xi)|)}{\mathbf F}(\xi)\big]=1,\label{eq1}
\\
&L(\alpha):=\E^{\mu_\alpha\otimes\Pi_\alpha}\big[\e^{-\lambda(\alpha)(|\xi'|+|\mathcal J(\xi)|)}(|\xi'|+|\mathcal J(\xi)|){\mathbf F}(\xi)]<\infty.\label{eq2}
\end{align}
\end{theorem} 
\noindent{\bf Notation: The birth-death process $\Gamma_{\alpha,T}$ on the entire period $[-T,T]$:} Before presenting the proof of Theorem \ref{thm2.6}, let us first we recollect some notation from Section \ref{sec-2}, where we introduced the Poisson point process (PPP) $\Gamma_{\alpha,T}$ with intensity measure $\gamma_{\alpha,T}(\d s \d t):=\alpha \e^{-(t-s)} \,\ 1_{-T\leq s<t \leq T} \,\, \d s \, \d t$. %Each realization of this PPP consists of $n$ of possibly overlapping intervals $\{[s_i,t_i]\}_{i=1}^n$ in $[-T,T]$. The union of these intervals need not be connected, 
%and will be a union of disjoint intervals, with gaps in-between and each interval being a union of overlapping sub-intervals $\{[s_i,t_i]\}_{i=1}^{n_r}$, with $n=\sum_r n_r$. 
%We called each $\xi_r=\{[s_i,t_i]\}_{i=1}^{n_r}$ an {\it active period}, or a {\it{cluster}}, and these clusters are separated by {\it gaps}, or {\it dormant periods} denoted by $\xi^\prime_r$. Dormant and active periods alternate, beginning and ending with dormant periods $\xi^\prime_1=[-T,\min_i s_i]$ and $\xi^\prime_k=[\max_i t_i,T]$, splitting the interval $[-T,T]$  into a collection $\widehat\xi=\big\{\xi^\prime_1,\xi_1,\dots,\xi^\prime_{k^\star(T)-1},\xi_{k^\star(T)}, \xi_{k^\star(T)+1}^\prime\big\}$
%of $k^\star(T)+1$ dormant intervals $\{\xi^\prime_{r}\}$ and $k^\star(T)$ active intervals $\{\xi_{r}\}$. If $\mathcal J(\xi)= [\min_i s_i, \max_i t_i]$ denotes the time span of an active period $\xi=\{[s_i,t_i]\}_{i=1}^{n(\xi)}$, then we clearly have $2T=  \sum_{j=1}^{k^\star(T)+1}|\xi_{j}^\prime|+\sum_{j=1}^{k^\star(T)}|\mathcal J(\xi_j)|$.
The above PPP also can be recast as a birth-death process on the entire time span $[-T,T]$ whose law is also denoted by $\Gamma_{\alpha,T}$,
and any realization of which starts with zero population size $n(-T)=0$ at time $-T$, with the population conditioned to die out at the terminal time $T$ (i.e., conditioned on event $\{n(T)=0\}$). Under $\Gamma_{\alpha,T}$, the birth rate at time $s$ is $b_{\alpha,T}(s)=\alpha(1-e^{T-s})$, while its death rate at time $t$ is $d_{\alpha,T}(t)=(1-e^{-(T-t)})^{-1}$. Note that at any given time $s$, both birth and death rates depend {\it only} on  $T-s$,  which is the {\it duration of time left}. In other words, 
the law $\Pi_{\alpha,T}$ defined in Section \ref{sec-Q-alpha-T} is the restriction of $\Gamma_{\alpha,T}$ on the first excursion of a single active period.

Also, with $\gamma_{\alpha,T}(\d s \d t)=\alpha \e^{-(t-s)} \,\ 1_{-T\leq s<t \leq T} \,\, \d s \, \d t$  we have %The intervals $J_i=[s_i, t_i]$, $i=1,\ldots, n(\xi)$ are the life times of individuals in the active period $\xi$ and  $\cup_i J_i$
% is the set of times when the population size $n(t)>0$.  The population size $n(t)$ starts from $0$ at time $-T$ and successively 
% enters  active (when $n(t)>0$) and dormant (when $n(t)=0$) periods.  In the interval  $[-T, T]$ there are a random number $k^\star(T)+1$  of dormant periods $\{\xi^\prime_1,\dots,\xi^\prime_{k^\star(T)+1}\}$ and $k^\star(T)$ active periods $\{\xi_1,\dots,\xi_{k^\star(T)}\}$, with $\mathcal J(\xi_j)=\cup_{i=1}^{n(\xi_j)}[s_i,t_i]$ itself being an interval (without any gaps) for the $j$th cluster $\xi_i$ so that 
 \begin{equation}\label{cdef2}
 \begin{aligned}
%2T= \sum_{j=1}^{k^\star(T)+1} |\xi^\prime_j| + \sum_{j=1}^{k^\star(T)} |\mathcal J(\xi_j)|,\quad\mbox{and  }
\alpha c(T)&= \int\int \gamma_{\alpha,T}(\d s \d t)= \alpha \int\int_{-T\leq s < t \leq T} \e^{-(t-s)} \d s \d t\\
&=\alpha\int_{-T}^T  (1-\e^{-(T-s)}) \d s=2\alpha T+o(T),
\end{aligned} %\quad \mbox{as }T\to\infty,
\end{equation}
as $T\to\infty$. %where $\alpha c(T)= \int\int \gamma_{\alpha,T}(\d s \d t)= \alpha \int\int_{-T\leq s < t \leq T} \e^{-(t-s)} \d s \d t$.% so that $\alpha c(T)$ is the expected number of total births in $[-T,T]$ under $\Theta_{\alpha,T}$. %The restriction of the law $\mathbf Q_{\alpha, T}$ of this birth-death process on the excursion or history of the {\it first} active period in $[-T,T]$ (starting with population size $1$) is denoted by $\Pi_{\alpha,T}$ (see \cite[Remark 4.9 and Section 4.2]{MV18} for details)
Then \eqref{cdef2}, together with the rewrite of the Polaron measure $\widehat\P_{\alpha,T}$ in the proof of Theorem \ref{thm1}
imply %$\E^{\Theta_{\alpha,T}}[\frac{\e^{\alpha c(T)}}{Z_{\alpha,T}} \prod_{j=1}^{k^\star(T)} \mathbf F(\xi_j)]=1$. Thus, by %\E^{\Theta_{\alpha,T}}\bigg[\e^{-2T(g(\alpha)-\alpha)} \prod_{j=1}^{k^\star(T)} \mathbf F(\xi_j)\bigg] = \e^{o(T)}\quad \mbox{as  }T\to\infty.
%\end{equation}
%The second assertion above is a consequence of \eqref{varfor} and (the second statement in) \eqref{cdef2}. Then, from the first statement of \eqref{cdef2} it also follows that
\begin{equation}\label{1}
\begin{aligned}
&\E^{\Gamma_{\alpha,T}}\Big[\frac{\e^{\alpha c(T)}}{Z_{\alpha,T}} \prod_{j=1}^{k^\star(T)} \mathbf F(\xi_j)\Big]=1, \quad\mbox{so that    }
\E^{\Gamma_{\alpha,T}}\big[\e^{-(g(\alpha)-\alpha)(2T)}\prod_{j=1}^{k^\star(T)} \mathbf F(\xi_j)\big] =\e^{o(T)}\end{aligned}
\end{equation}
as $T\to\infty$. In the second expression above, we used \eqref{cdef2} and the asymptotic behavior $g(\alpha)= \lim_{T\to\infty}\frac 1 {2T}\log Z_{\alpha,T}$ of the partition function $Z_{\alpha,T}$; recall \eqref{eq-DV-0}. We also remark that under $\Gamma_{\alpha, T}$ although the individual terms in the expectation in \eqref{1} are dependent, this dependence disappears as $T\to \infty$.
\begin{proof}[{\bf Proof of Theorem \ref{thm2.6}}]
\noindent{\bf Step 1:} First we prove that $q(\alpha) \leq 1$ in \eqref{eq1}. It will be convenient to slightly change the notation: we write $\Gamma_{\alpha, T}$ for the law of the birth and death process with the same birth rate $b_{\alpha,T}$ and death rate $d_{\alpha,T}$ defined before and with the same conditioning of dying out $\{n(T)=0\}$ at time $T$
but conditioned to start with zero population $n(0)=0$ at time zero (instead of conditioning on $\{n(-T)=0\}$), so that it has ``remaining time" $T$ (since both birth and death rates only depend on the remaining time, this change of notation is irrelevant). Let us define $\theta_1=\inf\{t: n(t)=1\}$ and $\sigma_1=\inf\{t: t> \tau \  {\rm and} \  n(t)=0\}$. We can recursively define $\theta_j,\sigma_j$ for $j=2,\dots, k^\star(T)$, and since we will have to end in a dormant period by time $T$, eventually $\theta_{k^\star(T)+1}$ will not exist (we will then say that $\theta_{k^\star(T)+1}=\infty$). Then writing $2T=\sum_{j=1}^{k^\star(T)+1}|\xi_{j}^\prime|+\sum_{j=1}^{k^\star(T)}|\mathcal J(\xi_j)|$ in \eqref{1}, we have
$$
\E^{\Gamma_{\alpha, T}}\bigg[\prod_{j=1}^{k^\star(T)}  \big [{\bf F}(\xi_j) \e^{-(\sigma_{j}-\sigma_{j-1})(g(\alpha)-\alpha)}\big] \e^{-(T-\sigma_{k^\star(T)})(g(\alpha)-\alpha)}\bigg]=\e^{o(T)}
$$
For any $A>0$ let us denote by ${\bf F}_A(\xi)=\min\{A, {\bf F}(\xi)\}$ the truncation of $\mathbf F$ at level $A$. 
Let $\delta>0$ be a (small) constant which
will be chosen later. Note that our renewal period consists of one dormant period, followed immediately by an active period. 
If  a renewal block starts at time $\theta_j<(1-\delta)T$, $\mathbf F$ (corresponding to the active period belonging to that renewal period) 
it will be truncated (i.e. will be replaced by $\mathbf F_A$), otherwise not. Let $\xi_{j^\ast(T)}$
 be the first  active period starting at a time (strictly) after $(1-\delta)T$. Then the above display implies 
\begin{equation}\label{2}
\begin{aligned}
&\E^{\Gamma_{\alpha, T}}\bigg[\prod_{j=1}^{j^\ast(T)-1}  \big [{\bf F}_ A(\xi_j)\e^{-(\sigma_{j}-\sigma_{j-1})(g(\alpha)-\alpha)}\big]\\
&\qquad\quad\times\prod_{j=j^\ast(T)}^{k^\star(T)}  \big[{\bf F}(\xi_j) \e^{-(\sigma_{j}-\sigma_{j-1})(g(\alpha)-\alpha)}\big] \, \e^{-(T-\sigma_{k^\star(T)})(g(\alpha)-\alpha)}\bigg]\le \e^{o(T)}. 
\end{aligned}
\end{equation}
We would like to provide a suitable lower bound for the expectation of the second product $\prod_{j=j^\star(T)}^{k^\star(T)}$ above. Note that we always have 
$-[\sum_{j=j^\star(T)}^{k^\star(T)} (\sigma_j-\sigma_{j-1}) + (T-\sigma_{k^\star(T)})][g(\alpha)-\alpha] \ge - C\delta T$ for some $C=C(\alpha)\in (0,\infty)$.  Now, by 
a similar lower bound as in the proof of Lemma \ref{lemma-sqrt2}\footnote{In the proof of Lemma \ref{lemma-sqrt2} (see below) we use that for any symmetric positive definite matrix $M=(m_{ij})$, $\mathrm{det}(M)\leq \prod_i m_{ii}$. In our case, $m_{ii}=u_i^2 |J_i| = u_i^2 \tau_i= u_i^2(t_i-s_i)$, and we obtain a lower bound 
$\Phi(\xi_j,\bar u) \geq \prod_{i=1}^{n(\xi_j)} (1+ u_i^2 \tau_i)^{-3/2}$. Using that $\int_0^\infty \frac{\d u}{(1+ u^2 \tau)^{-3/2}} =\frac 1 {\sqrt\tau}$ and that under $\Pi_\alpha$, the lifespan $\tau$ is exponentially distributed with parameter $1$, we obtain $\E^{\Pi_\alpha}[\frac 1{\sqrt\tau}]=\sqrt\pi$. In the present context, we can use the same argument together with the following observation to get a similar lower bound: Note that $\Pi_{\alpha,T}[\tau >t]= \e^{-\int_0^t d_{\alpha,T}(s) \d s}$, and 
 $d_{\alpha,T}(s)=\frac 1 {1- \e^{-(T-s)}}>1$. Now if $f(t)=\Pi_\alpha[\tau>t]$ and $g(t)= \Pi_{\alpha,T}[\tau>t]$, and since $d_{\alpha,T}(\cdot)>1$, we have 
$g(\cdot) \leq f(\cdot)$ and thus $-\int_0^\infty \frac 1{\sqrt\tau} \d g(\tau) \geq - \int_0^\infty \frac 1{\sqrt\tau} \d f(\tau)$.}
and by successive conditioning and the renewal property, we can obtain a suitable lower bound for the product $\prod_{j=j^\star(T)}^{k^\star(T)} \mathbf F(\xi_j)$. 
Combining these two lower bounds, we obtain from \eqref{2} that 
\begin{equation}\label{3}
\E^{\Gamma_{\alpha, T}}\Big[\prod_{j=1}^{j^\ast(T)-1}  \big [{\bf F}_ A(\xi_j)\e^{-(\sigma_{j}-\sigma_{j-1})(g(\alpha)-\alpha)}\big]\Big]\le \e^{C\delta T}. 
\end{equation}
Assume that in contrary to \eqref{eq1},  
$q(\alpha)=\E^{\Pi_\alpha\otimes\mu_\alpha}\big[\e^{-\lambda(\alpha)(|\xi^\prime |+|\mathcal J(\xi)|)}{\mathbf F}(\xi)\big]>1$. 
Recall that on any single cluster $\xi$, $\Pi_{\alpha,T}$ converges to $\Pi_\alpha$ as $T\to\infty$ in the total variation distance (see Remark \ref{rmk-bd-3}). Thus, for $T_0$ and $A$ sufficiently large, there is $q^\prime>1$ such that 
\begin{equation}\label{4}
\E^{\Gamma_{\alpha,T_0}}\big[\e^{-\lambda(\alpha)(|\xi^\prime |+|\mathcal J(\xi)|)}{\mathbf F_A}(\xi)\big]\geq q^\prime>1. 
\end{equation}
By successive conditioning and the renewal property, since $\delta T\to \infty$ for any fixed $\delta>0$, we obtain from \eqref{4} and \eqref{3} that 
$\E^{\Gamma_{\alpha,T}}[(q^\prime)^{j^\star(T)-1}] \leq \e^{C\delta T}$. But this leads to a contradiction because by the renewal theorem, there exists a constant $c_0\in (0,\infty)$ such that with high $\Gamma_{\alpha,T}$-probability as $T\to\infty$, we have $(j^\star(T) -1) \ge c_0 (1-\delta) T$. Given any $q^\prime$ we can choose $\delta>0$ sufficiently small to contradict \eqref{4}. Hence, for any $\alpha>0$, we have $q(\alpha)=\E^{\Pi_\alpha\otimes\mu_\alpha}\big[\e^{-[g(\alpha)-\alpha](|\xi^\prime |+|\mathcal J(\xi)|)}{\mathbf F}(\xi)\big]\leq 1$.

 \noindent{\bf Step 2:} In this step we will prove that under the tilted mixing measure $\widehat{\Theta}_{\alpha,T}$ defined below, the expected total length of all the dormant periods 
 is $O(T)$ (shown in \eqref{eq3} below) and that with positive $\widehat{\Theta}_{\alpha,T}$-probability there are also $O(T)$ many renewals (shown in \eqref{renewals}). 
 Let $f(s,t;\tau)$ be the indicator function of $(s,t;\tau)$ such that 
 $\tau\in (s,t)$ 
and  $W(\tau)=\sum_i f(s_i,t_i;\tau)$ is the number  
of intervals that cover $\tau$. We are interested in the probability distribution of $W$ under the tilted mixing measure (recall \eqref{eq-Theta-hat-T})
\begin{equation}\label{hatQ}
\widehat{\Theta}_{\alpha,T}(\d\hat\xi\d\hat u)= \frac{\e^{\alpha c(T)}}{Z_{\alpha,T}}\prod_{j=1}^{k^\star(T)}\bigg[\bigg(\sqrt{\frac2\pi}\bigg)^{n(\xi_j)} \Phi(\xi_j,\bar u_j)\d\bar u_j\bigg] \Gamma_{\alpha,T}(\d\hat\xi). 
\end{equation}
on the interval $[-T,T]$.

For this it suffices to compute, for any $\sigma>0$, 
\begin{align*}
&\E^{\widehat{\Theta}_{\alpha, T}}[\exp[-\sigma\sum_i f(s_i,t_i;\tau)]]\\
&= {\e^{\alpha c(T)}\over Z_{\alpha, T}} \E^{\Theta_{\alpha, T}}\E^\P\bigg[\bigg(\sqrt{\frac2\pi}\bigg)^{n(\hat\xi)} \int_{(0,\infty)^{n(\hat\xi)}}\d u_1\dots \d u_{n(\hat\xi)} \\
&\qquad\qquad\qquad\qquad\times \exp\bigg(-\sigma\sum_{i=1}^{n(\hat\xi)} f(s_i,t_i;\tau) -{1\over 2}\sum_{i=1}^{n(\hat\xi)} u_i^2 |\omega(t_i)-\omega(s_i)|^2\bigg)\bigg]
\\
&= {\e^{\alpha c(T)}\over Z(\alpha, T)} \E^{\Theta_{\alpha, T}}\E^\P\bigg[\prod_{i=1}^{n(\hat\xi)} \frac {\e^{-\sigma f(s_i,t_i;\tau)}}{|\omega(s_i)-\omega(t_i)|}     \bigg]
\end{align*}
where we wrote $n(\hat\xi)=\sum_j n(\xi_j)$ and in the last identity we used $|x|^{-1}=\sqrt{2/\pi}\int_0^\infty \d u \e^{-u^2|x|^2/2}$. Thus, 
\begin{equation}\label{eqPoisson}
\begin{aligned}
&\E^{\widehat{\Theta}_{\alpha, T}}\big[\exp[-\sigma\sum_i f(s_i,t_i,\tau)]\big]\\
&={1\over Z(\alpha,T)}\E^\P\bigg[ \sum_{n\geq 0}{\alpha^n\over n!} \prod_{i=1}^n \int\cdots\int_{-T\le s_i<t_i\le T}  \d s_i \d t_i \e^{-| t_i-s_i|} \frac {\e^{-\sigma f(s_i,t_i;\tau)}}{|\omega(s_i)-\omega(t_i)|}\bigg]\cr
&={1\over Z(\alpha,T)}\E^\P\bigg[\sum_{n\geq 0} {\alpha^n\over n!} \bigg(\int\int_{-T\le s<t\le T}\e^{-|t-s|}\frac {\e^{-\sigma f(s,t;\tau)}}{|\omega(t)-\omega(s)|}\d s\d t\bigg)^n \bigg]\cr
&=\E^{\widehat{\P}_{\alpha,T}}\bigg[\exp \bigg(-\alpha\int\int_{-T\le s<t\le T}\frac {\e^{-|t-s|}(1-\e^{-\sigma f(s,t;\tau)})} {|\omega(t)-\omega(s)|}\d s\d t\bigg)\bigg]\cr
\end{aligned}
\end{equation}
where $\widehat\P_{\alpha,T}$ is the Polaron path measure. Thus, the number $W(\tau)$ of intervals that cover $\tau$ is Poisson-distributed under $\widehat{\Theta}_{\alpha,T}$ 
with a random intensity parameter 
$$
\Lambda_{\alpha,T}^{\ssup\tau}(\omega)=\alpha\int\int_{-T\le s<\tau<t \le T}\frac {\e^{-|t-s|}}{|\omega(t)-\omega(s)|} \d t\d s.
$$

Let $\Lambda_{\alpha,T}=\frac 1{2T}\int_{-T}^T \Lambda_{\alpha,T}^{\ssup\tau} \d\tau$. Then 
\begin{equation}\label{derivative}
\begin{aligned}
\E^{\widehat\P_{\alpha,T}}[\Lambda_{\alpha,T}]&=\frac 1 {2T} \E^{\widehat\P_{\alpha,T}}\bigg[\int_{-T\leq s < t \leq T} \frac{\d s \d t \e^{-|t-s|} |t-s|}{|\omega(s)-\omega(t)|}\bigg] \\
&= \frac{\d}{\d\delta}\bigg(\frac 1 {2T} \log \E\bigg[\exp\bigg( \alpha \int_{-T\leq s < t \leq T} \frac{\e^{-(1+\delta)|t-s|}\d s \d t}{|\omega(s)-\omega(t)|}\bigg)\bigg]\bigg)\bigg|_{\delta=0}
\end{aligned}
\end{equation}
On the other hand, by Brownian rescaling and using the convexity of $\alpha\mapsto g(\alpha)=\lim_{T\to\infty} \frac 1 {2T}\log Z_{\alpha,T}\in (0,\infty)$ (recall \eqref{eq-DV-0}), 
we have $\sup_{T>0} \E^{\widehat\P_{\alpha,T}}[\Lambda_{\alpha,T}] =: C(\alpha)<\infty$. Since the number $W(\tau)$ of intervals that cover $\tau$ is Poisson-distributed under $\widehat{\Theta}_{\alpha,T}$ with parameter $\Lambda_{\alpha,T}^{\ssup\tau}$, we have by Jensen's inequality, 
\begin{equation}\label{eq3}
\liminf_{T\to\infty} \frac 1 {2T}\int_{-T}^T \d\tau \widehat{\Theta}_{\alpha,T}[W(\tau)=0] \geq \e^{-C(\alpha)}:= c_0(\alpha)>0.
\end{equation}
That is, under $\widehat{\Theta}_{\alpha,T}$, the expected total length of all the dormant periods is at least $c_0(\alpha)(2T)$. Using this, we want to show that 
with positive $\widehat{\Theta}_{\alpha,T}$-probability as $T\to\infty$, there are at least $O(T)$ dormant periods, that is, we will show that 
there is $c_1(\alpha)>0$ such that
\begin{equation}\label{renewals}
\liminf_{T\to\infty}\widehat{\Theta}_{\alpha,T}[k^\star(T) \geq c_1(\alpha) T] >0.
\end{equation} 
For this purpose, we will show that there is $\ell_0$ such that
the expected sum of the lengths of dormant periods larger than $\ell_0$ does not exceed $\frac 12 c_0(\alpha)(2T)$ for $T\geq T_0$, which 
will imply that with positive $\widehat{\Theta}_{\alpha,T}$-probability, there are at least $\frac 1 {\ell_0} \frac{c_0(\alpha)}2(2T)$ many dormant periods. Thus, it suffices to show that the expected sum of the lengths of dormant periods larger than $\ell$ is at most $\theta(\ell)(2T)$, and then we need to control $\int_{\ell}^\infty \sigma \d\theta(\sigma)$ or $\int_\ell^\infty \theta(\sigma) \d\sigma$ and choose $\ell_0$ such that $\int_{\ell_0}^\infty \theta(\sigma)\d\sigma \leq \frac 12 c_0(\alpha)(2T)$.

For this purpose, let us divide the interval $[-T,T]$ into $2T/\ell$ intervals $\{V_j\}_{j=1}^{2T/\ell}$ of length $\ell$. Any dormant period of length $2\ell$
in $[-T,T]$ will contain fully some $V_j$.  If in $[-T,T]$ there are  $N$ dormant periods of length at least $2\ell$, then  out of the $2T/\ell$ intervals $\{V_j\}$ of length $\ell$ at least $N$ are dormant. Therefore, it suffices to estimate the expected number of dormant intervals amongst these 
$2T/\ell$ intervals of length $\ell$. Thus, to prove \eqref{renewals} it suffices to show that, 
\begin{equation}\label{eq4}
\sup_T \frac 1 {2T} \sum_{j=1}^{2T/\ell} \widehat{\Theta}_{\alpha,T}[V_j\,\,\mbox{is dormant }] \leq \theta(\ell), \quad\mbox{such that   }\int_0^\infty \theta(\sigma)\d\sigma <\infty. 
\end{equation}
But as before (recall \eqref{eqPoisson}) if $W(V)$ denotes the number of intervals $\{[s_i,t_i]\}_i$ which intersect a given interval $V$, then under $\widehat{\Theta}_{\alpha,T}$, 
$W(V)$ is Poisson-distributed with a random intensity parameter $\Lambda^{\ssup V}_{\alpha,T}(\omega)= \alpha\int\int_{[s,t] \cap V\ne \emptyset} \frac{\e^{-(t-s)} \d s \d t}{|\omega(t)-\omega(s)|}$, and thus, 
\begin{equation*}
\begin{aligned}
\widehat{\Theta}_{\alpha,T}[V_j\,\,\mbox{is dormant }] &\leq \E^{\widehat\P_{\alpha,T}}\bigg[\exp\bigg(-\alpha\int\int_{[s,t] \cap V_j\ne \emptyset} \frac{\e^{-(t-s)} \d s \d t}{|\omega(t)-\omega(s)|}\bigg)\bigg]
\\
&\leq \E^{\widehat\P_{\alpha,T}}\bigg[\exp\bigg(-\alpha\sum_{i: [i,i+1]\subset V_j}\int_i^{i+1}\int_i^{i+1}\frac{\e^{-(t-s)} \d s \d t}{|\omega(t)-\omega(s)|}\bigg)\bigg]
\\
&\leq \E^{\widehat\P_{\alpha,T}}\bigg[\exp\bigg(-\alpha \e^{-1}\sum_{i: [i,i+1]\subset V_j}\int_i^{i+1}\int_i^{i+1}\frac{ \d s \d t}{|\omega(t)-\omega(s)|}\bigg)\bigg]
\end{aligned}
\end{equation*}
To show \eqref{eq4}, we estimate, for any $\lambda>0$, 
\begin{equation}\label{eq6}
\begin{aligned}
&\frac 1 {2T}\E^{\widehat\P_{\alpha,T}}\bigg[\sum_{j=-T/\ell}^{T/\ell}\exp\bigg(-\alpha \e^{-1}\sum_{i: [i,i+1]\subset V_j}\int_i^{i+1}\int_i^{i+1}\frac{ \d s \d t}{|\omega(t)-\omega(s)|}\bigg)\bigg]
\\
&\leq \frac 1{2\lambda T} \log \E^{\P}\bigg[\exp\bigg(\lambda \sum_{j=-T/ \ell}^{T/ \ell}\e^{- \alpha \e^{-1}\sum_{i=j\ell}^{(j+1)\ell-1} \int_i^{i+1}\int_i^{i+1}\frac{ \d s \d t}{|\omega(t)-\omega(s)|}} \bigg)\bigg] + \frac 1 {2\lambda T} H(\widehat\P_{\alpha,T}| \P)\\
&\leq \frac 1{\lambda \ell} \log \E^{\P}\bigg[\exp\bigg(\lambda  \e^{- \alpha \e^{-1}\sum_{i=0}^{\ell-1} \int_i^{i+1}\int_i^{i+1}\frac{ \d s \d t}{|\omega(t)-\omega(s)|}} \bigg)\bigg] + \frac C\lambda,
\end{aligned}
\end{equation}
where in the first estimate above we applied the relative entropy inequality, while in the second estimate we used $\sup_T \frac 1 {2T} H(\widehat\P_{\alpha,T}| \P)\leq C$.\footnote{The relative entropy estimate states that for any two probability measures $\mu$ and $\nu$ and any bounded measurable function $f$, for any $\lambda>0$, $\E^\nu[f] \leq \frac 1\lambda \log \E^\mu[\e^{\lambda f}]+ \frac 1 \lambda H(\nu| \mu)$. Presently the singularity of the Coulomb potential does not cause any trouble as we may replace $\frac 1 {|x|}$ by $\frac 1{1+|x|}$ to get the required upper bound. Also, to show $\sup_T \frac 1 {2T} H(\widehat\P_{\alpha,T}| \P)\leq C$ in the next step,  we have again used, just as in \eqref{derivative} 
that $\frac 1{2T} H(\widehat\P_{\alpha,T}|\P)= \frac 1 {2T}\E^{\widehat\P_{\alpha,T}}[\alpha \int\int \frac{\e^{-(t-s)\d s \d t}}{|\omega(t)-\omega(s)|}]- \frac 1 {2T}\log Z_{\alpha,T} = \frac{\d}{\d\delta}\big(\frac 1 {2T} \log \E^{\P}\big[\exp\big\{\alpha(1+\delta)\int\int \frac{\e^{-(t-s)\d s \d t}}{|\omega(t)-\omega(s)|}\big\}\big]\big)\big| _{\delta=0}- \frac 1 {2T}\log Z_{\alpha,T}$.} To show \eqref{eq4}, it suffices to choose $\lambda= \ell^2$ in the last expression of \eqref{eq6} and prove that 
$\sup_\ell \E^\P[\exp(\ell^2 \e^{-(\alpha/\e) \sum_{i=0}^{\ell-1} X_i})]<\infty$ with $X_i(\omega):= \int_0^1\int_0^1 \frac{\d s \d t}{|\omega(s)- \omega(t)|} \geq 0$ and $X_1,\dots, X_\ell$ are i.i.d. random variables. Since $x\mapsto 1/x$ is convex, for any $\delta>0$, we have by Gaussian tail estimate 
$\P[X_i \leq \delta] \leq \P[\int_0^1\int_0^1 {\d s \d t}{|\omega(s)-\omega(t)|} \geq \frac 1\delta\big] \leq \e^{-\frac k{\delta^2}}$ for some $k>0$. 
If we set $A:=\{\omega\colon\#\{i\colon X_i(\omega)\geq \delta\}\geq \ell/2\}$ and estimate $\E^\P[\exp(\ell^2 \e^{-(\alpha/\e) \sum_{i=0}^{\ell-1} X_i})]$ by splitting the integrand 
over $A$ and $A^c$, then it follows readily that $\sup_\ell \E^\P[\exp(\ell^2 \e^{-(\alpha/\e) \sum_{i=0}^{\ell-1} X_i})] < \infty$ if
we choose $\delta=\frac 1 {C\sqrt\ell}$ with $C>0$ sufficiently large. 
Thus \eqref{eq4} is shown which also completes the proof of \eqref{renewals}.

\noindent{\bf Step 3:} Finally, we prove the identity \eqref{eq1} and the estimate \eqref{eq2}. To show \eqref{eq1}, 
recall that in Step 1 we have already shown that 
$q(\alpha) \leq 1$. Let $\mathbb Q_\alpha$ be the law of the 
birth and death process starting with zero individuals at time $0$, with a birth rate of $\alpha>0$ and death rate of $1$ for each person in the population. In other words 
if $n$ is the current size of the population, the transition rates are $n\to n+1$ with rate $\alpha$ and $n\to n-1$ with rate $n$ provided $n>0$. It is a positive recurrent process with the Poisson distribution  with parameter $\alpha$ as the invariant measure. Note that $\Gamma_{\alpha,T}[\cdot]= \mathbb Q_\alpha[\cdot| n(T)=0]$ and 
there exists $p>0$ such that $\mathbb Q_{\alpha}[n(T)=0]\ge p$ for all $T$. Recall \eqref{hatQ}, and that $\sigma_{k^\star(T)}$ is the time when the last active period dies out before time $T$, also recall from \eqref{1} that $\frac{\e^{\alpha c(T)}}{Z_{\alpha,T}} \e^{2T(g(\alpha)-\alpha)}= \e^{o(T)}$ as $T\to\infty$. Then for any $c>0$,
$$
\begin{aligned}
&\widehat{\Theta}_{\alpha,T}[ k^\star(T) \geq c T] \\
&= \e^{o(T)} \E^{\Gamma_{\alpha,T}} \bigg[ \prod_{j=1}^{k^\star(T)} \bigg( \e^{-[g(\alpha)-\alpha] [ |\mathcal J(\xi_j)| + |\xi^\prime_j| ] } \mathbf F(\xi_j)\bigg) \e^{-[g(\alpha)-\alpha][T-\sigma_{k^\star(T)}]} \1\{k^\star(T) \geq c T\} \bigg] \\
&\leq \frac 1 p \e^{o(T)} \E^{\mathbb Q_\alpha}\bigg[ \prod_{j=1}^{k^\star(T)} \bigg( \e^{-[g(\alpha)-\alpha] [ |\mathcal J(\xi_j)| + |\xi^\prime_j| ] } \mathbf F(\xi_j)\bigg) \e^{-[g(\alpha)-\alpha][T-\sigma_{k^\star(T)}]} \\
&\qquad\qquad\qquad\qquad\qquad\1\{k^\star(T) \geq c T, \, \sigma_{k^\star(T)} < T; \, n(t)=0 \,\forall t\in [\sigma_{k^\star(T)},T]\big\}\bigg] \\
&= \frac 1 p \e^{o(T)} \E^{\mathbb Q_\alpha} \bigg[\prod_{j=1}^{k^\star(T)} \bigg( \e^{-[g(\alpha)-\alpha] [ |\mathcal J(\xi_j)| + |\xi^\prime_j| ] } \mathbf F(\xi_j)\bigg) \e^{-[g(\alpha)-\alpha][T-\sigma_{k^\star(T)}]} \\
&\qquad\qquad\qquad\qquad\times\e^{-\alpha(T-\sigma_{k^\star(T)})} \1\{k^\star(T) \geq c T\}  \1\{\sigma_{k^\star(T)} < T\}\bigg] \\
&=  \frac 1 p \e^{o(T)} \E^{\mathbb Q_\alpha} \bigg[\prod_{j=1}^{k^\star(T)} \bigg( \e^{-[g(\alpha)-\alpha] [ |\mathcal J(\xi_j)| + |\xi^\prime_j| ] } \mathbf F(\xi_j)\bigg) \\
&\qquad\qquad\qquad\qquad\times\e^{-g(\alpha)[T-\sigma_{k^\star(T)}]} \1\{k^\star(T) \geq c T\} \, \1\{\sigma_{k^\star(T)} < T\}\bigg] \\
&= \frac 1 p \e^{o(T)} \sum_{k=cT}^\infty \E^{\mathbb Q_\alpha} \bigg[\prod_{j=1}^{k} \bigg( \e^{-[g(\alpha)-\alpha] [ |\mathcal J(\xi_j)| + |\xi^\prime_j| ] } \mathbf F(\xi_j)\bigg) \e^{-g(\alpha)[T-\sigma_{k}]} \, \1\{\sigma_{k} < T\}\bigg] \\
&\leq \frac 1 p \e^{o(T)} \sum_{k=cT}^\infty \E^{\mathbb Q_\alpha} \bigg[\prod_{j=1}^{k} \bigg( \e^{-[g(\alpha)-\alpha] [ |\mathcal J(\xi_j)| + |\xi^\prime_j| ] } \mathbf F(\xi_j)\bigg)\bigg] 
\leq \frac 1 p \e^{o(T)} \frac{[q(\alpha)]^{c T}}{1-q(\alpha)}.
\end{aligned}
$$
But if 
$$
q(\alpha)=\E^{\mu_\alpha\otimes\Pi_\alpha}\big[\e^{-[g(\alpha)-\alpha](|\xi^\prime |+|\mathcal J(\xi)|)}{\mathbf F}(\xi)\big]= \E^{\mathbb Q_\alpha}\big[\e^{-[g(\alpha)-\alpha](|\xi^\prime |+|\mathcal J(\xi)|)}{\mathbf F}(\xi)\big]<1,
$$ 
 the above estimate would imply that $\widehat{\Theta}_{\alpha,T}[ k^\star(T) \geq c T]\to 0$ for any $c>0$, which would contradict \eqref{renewals}. Thus, $q(\alpha)=1$ and consequently, the identity \eqref{eq1} must hold. To show \eqref{eq2} we again recall \eqref{eq3}. Note that by the renewal theorem,\footnote{By the renewal theorem, if the relevant expectation is infinite, then the left hand side of \eqref{eq3} converges to zero.} the left hand side of \eqref{eq3} is the ratio (recall the definition of $L(\alpha)$ from \eqref{eq2} and that $\mu_\alpha$ is exponentially distributed with parameter $\alpha$ on a dormant period $\xi^\prime$)
\begin{align*}
\frac{\E^{\mu_\alpha}[|\xi^\prime|\e^{-[g(\alpha)-\alpha)]|\xi^\prime|}]}{\E^{\mu_\alpha}[\e^{-[g(\alpha)-\alpha)]|\xi^\prime|}]}     \big[ {L(\alpha)} \big]^{-1} 
= \frac{[\alpha/g(\alpha)^2]}{[\alpha/g(\alpha)]} [L(\alpha)]^{-1}=\frac{[1/g(\alpha)]}{L(\alpha)}.
\end{align*}
Since $\frac 1{g(\alpha)}$ is finite, and by \eqref{eq2} the above ratio is strictly positive, 
we must have $L(\alpha)<\infty$, which proves \eqref{eq2}.
\end{proof}

The lower bound on $\Phi$ from the following lemma was used in the Step 1 of the proof of Theorem \ref{thm2.6}. 

\begin{lemma}\label{lemma-sqrt2}
Fix any $\alpha>0$. Then for any active period $\xi$ with $n(\xi)\geq 1$ individuals, 
$$
\E^{\Pi_\alpha}\big[\mathbf F(\xi)\big] \ge \sqrt{2}>1
$$
 \end{lemma}
 \begin{proof}
 we first use the fact that for any symmetric positive definite matrix $M=(m_{ij})$, 
 $\mathrm{det}(M) \leq \prod_{i=1}^n m_{ii}$. In our case $m_{ii}=u_i^2|J_i|=u_i^2\tau_i$ .
 $$
\Phi(\xi,\bar u) \geq \prod_{i=1}^{n(\xi)}\frac{1}{(1+u_i^2\tau_i)^\frac32}
 $$
 where $\tau_1,\dots,\tau_{n}$ are exponentials with parameter $1$.  Since
$$
\int_0^\infty \frac {\d u}{(1+u^2\tau)^\frac{3}{2}}=\frac{1}{\sqrt \tau}
$$
and 
$$
\E^{\Pi_\alpha}\bigg[\frac{1}{\sqrt\tau}\bigg]=\int_0^\infty \frac {\e^{-\tau}}{\sqrt\tau}\d\tau=\sqrt{\pi},
$$
and thus, %because $n(\xi),\tau_1,\ldots,\tau_{n(\xi)}$ are all mutually independent we only need to  examine
\begin{equation}\label{eq-lemma-sqrt2}
\begin{aligned}
\E^{\Pi_\alpha}\bigg[\frac { (\sqrt{2/\pi})^{n(\xi)}}{\sqrt {\tau_1\cdots\tau_{n(\xi)}}}\bigg]&=\E^{\Pi_\alpha}\bigg[ \bigg(\sqrt{\frac 2\pi}\bigg)^{n(\xi)}\,(\sqrt{\pi})^{n(\xi)}\bigg]\\
&=\E^{\Pi_\alpha}\bigg[\e^{\frac{1}{2}n(\xi)\log 2}\bigg] \geq \sqrt 2 >1,
\end{aligned}
\end{equation}
since $n(\xi)\geq 1$.
\end{proof}

 \section{Identification of the limiting Polaron measure, its mixing properties and the central limit theorem}\label{sec-4}

In this section we will state and prove the three main results  (announced in Section \ref{sec-outline}) concerning the asymptotic behavior of the Polaron measures $\widehat\P_{\alpha,T}$ as $T\to\infty$. 

We fix $\alpha>0$ and in what follows we will write (and recall Theorem \ref{thm2.6})
\begin{equation}\label{eq2-thm2.5}
\begin{aligned}
&\lambda=\lambda(\alpha)= g(\alpha)-\alpha, \qquad\qquad\mbox{with      }g(\alpha) \mbox{     defined in  }\eqref{eq-DV-0},\\
&q(\alpha)=\E^{\Pi_{\alpha}\otimes\mu_\alpha}\bigg[\e^{-\lambda [|{\mathcal J}(\xi)|+|\xi'|]} \mathbf {\bf F}(\xi)\bigg]=1,\\
&L(\alpha)=\E^{\Pi_{\alpha}\otimes\mu_\alpha}\bigg[\e^{-\lambda [|{\mathcal J}(\xi)|+|\xi'|]} \mathbf {\bf F}(\xi) \,\,[|\mathcal J(\xi)|+|\xi^\prime|]\bigg]<\infty.
\end{aligned}
\end{equation}

\bigskip

\begin{theorem}[Identification of the limiting Polaron Measure]\label{thm4}
Fix any $\alpha>0$. On any finite interval $J=[-A,A]$, and on the $\sigma$-field generated  by the differences $\omega(t)-\omega(s)$ 
with $-A \le s< t \le A$,  the restriction ${\widehat \P}_{\alpha,T}^{\ssup J}$ of the Polaron measure ${\widehat \P}_{\alpha,T}$ on $J$ converges in total variation to the restriction to $J$ of 
\begin{equation}\label{eq-Polaron-limit}
\widehat\P_{\alpha}=\int \mathbf P_{\hat \xi, \hat u} (\cdot) \,\, \widehat{\mathbb Q}_\alpha(\d\hat\xi\,\d\hat u).
\end{equation}
In the above expression, $\widehat{\mathbb Q}_\alpha$ is the aforementioned stationary version of the process obtained by alternating the distribution 
\begin{equation}\label{eq-Pi-hat}
\widehat \Pi_{\alpha}(\d\xi\, \d \overline u)= \frac\alpha{\alpha+\lambda} \e^{-\lambda|\mathcal J(\xi)|} \bigg[ \bigg(\sqrt{\frac 2\pi}\bigg)^{n(\xi)}\,\, { \Phi(\xi,\overline u)} \d\bar u\bigg] \,\, \Pi_\alpha(\d\xi),
\end{equation}
on active periods $(\xi,\bar u)$ and the distribution 
\begin{equation}\label{eq-mu-hat}
\widehat\mu_\alpha(\d\xi^\prime)=\frac{\alpha + \lambda}\alpha \e^{-\lambda|\xi^\prime|} \mu_\alpha(\d\xi^\prime)
\end{equation}
on dormant periods $\xi^\prime$, with $\mu_\alpha$ denoting the exponential distribution with parameter $\alpha>0$. 
\end{theorem}

As a corollary to Theorem \ref{thm4}, we have the following central limit theorem that also provides an expression for the variance. 

\begin{theorem}[The central limit theorem for the Polaron]\label{thm5}
Fix any $\alpha>0$, and let $\widehat\nu_{\alpha,T}$ be the distribution of the rescaled increment 
\begin{equation}\label{increment}
\frac{1}{\sqrt {2T}}\big(\omega(T)-\omega(-T)\big)
\end{equation}
 under the finite-volume Polaron measure $\widehat\P_{\alpha,T}$, and let $\widehat\mu_{\alpha,T}$ be the 
 distribution of the same increments \eqref{increment} under the infinite-volume limit $\widehat\P_\alpha=\lim_{T\to\infty}\widehat\P_{\alpha,T}$ defined in \eqref{eq-Polaron-limit}.
 Then, for any $\alpha$ and as $T\to\infty$, both $\widehat\nu_{\alpha,T}$ and $\widehat\mu_{\alpha,T}$ converge to a centered three dimensional Gaussian law with covariance matrix given by $\sigma^2(\alpha)\,I$, where for any unit vector $v\in\R^3$, 
\begin{equation}
\begin{aligned}
\sigma^2(\alpha)&= \lim_{T\to\infty} \frac 1 {2T}\E^{\widehat{\P}_{\alpha,T}}\bigg[\big\langle v, \omega(T)-\omega(-T)\big\rangle^2\bigg]
= \frac{(\lambda+\alpha)^{-1}+ \Gamma(\alpha)}{(\lambda+\alpha)^{-1}
+L(\alpha)}\\
& =\frac{g(\alpha)^{-1}+ \Gamma(\alpha)}{g(\alpha)^{-1}
+L(\alpha)}\in (0,1) 
\end{aligned}
\end{equation}
and 
$$
\begin{aligned}
&\Gamma(\alpha)= \E^{\widehat \Pi_\alpha}\bigg[\E^{\mathbf P_{\xi,\bar u}} \big[\big\langle v, \omega(\sigma^\star)-\omega(0)\big\rangle^2\big]\bigg] \qquad \mbox{and}\,\,\,\,
L(\alpha)=\E^{\widehat{\Pi}_\alpha}\big[|{\mathcal J}(\xi)|\big]
\end{aligned}
$$
where $\mathbf P_{\xi,\bar u}$ is the Gaussian measure defined on increments corresponding to the quadratic form $\mathcal Q_{\xi,\bar u}$ attached to the excursion $(\xi,\bar u)$ of a single active period $\xi$ with time span $\mathcal J(\xi)=[0,\sigma^\star]$ (recall \eqref{eq-Q-xi-u-2}). \end{theorem}

Finally, the following result provides an exponential mixing property of the limiting Polaron measure $\widehat\P_\alpha$. 

\begin{theorem}[The mixing property of the limiting Polaron]\label{thm4.5}
Let us denote by ${\widehat\P}_{\alpha}^{\ssup J}$ the restriction of ${\widehat\P}_\alpha$ to increments in the interval $J$. There exists $\alpha_0\in (0,\infty)$ such that for any $\alpha\in (0,\alpha_0)$ there are constants $c(\alpha),C(\alpha)>0$ such that  restrictions to disjoint sets decay exponentially fast as the distance increases. In other words, for $A>0$, 
$$
\big\|{\widehat \P}_{\alpha}^{(-\infty,-A)\cup(A,\infty)}-{\widehat \P}_{\alpha}^{ (-\infty, -A)} \otimes {\widehat \P}_{\alpha}^{(A,\infty)}\big \|\le C(\alpha)\e^{-c(\alpha)A}
$$
\end{theorem}
\begin{remark}
The proof of Theorem \ref{thm4.5} depends on the expected size $\E^{\widehat\Pi_\alpha}[\e^{a|\mathcal J(\xi)|}]<\infty$ of a cluster having some exponential moment under the tilted birth-death process (see Lemma \ref{lemma-2-thm4.5}). The proof of this exponential moment requires $\alpha$ to be small and can be found in Theorem \ref{thm2} in the Appendix 
\footnote{The argument for Theorem \ref{thm2} depends on an explicit lower bound on the determinant appearing in the function $\Phi$. This lower bound works well as long as $\alpha$ remains small, but it does not seem to be useful when $\alpha$ is large.} While we do not claim to have a proof of an exponential moment of $|\mathcal J(\xi)|$ (resp. exponential mixing of $\widehat\P_\alpha$) for large $\alpha$, it is quite possible that such a statement for any $\alpha$ can be extracted from our method. As we will see in the proof of Theorem \ref{thm5}, the central limit theorem for the increments under $\widehat\P_\alpha$  will only need the ergodic theorem for the recurrent renewal process of dominant and active periods. 
\end{remark}

The proofs of the Theorem \ref{thm4}- Theorem \ref{thm4.5} can be found in Section \ref{sec-proofs-sec-4}. These proofs depend on deriving the law of large numbers for the mixing measure $\widehat\Theta_{\alpha,T}$ for the Polaron measure $\widehat\P_{\alpha,T}$. Section \ref{sec-renewal} is devoted to the derivation of this law of large numbers.

\subsection{Asymptotic behavior of the mixing measure $\widehat\Theta_{\alpha,T}$ as $T\to\infty$.}\label{sec-renewal}

In order to derive limiting assertions for $\widehat\Theta_{\alpha,T}$ we will need  to invoke some arguments based on renewal theory (\cite{F66}) and it is useful to collect  them at this point.

\begin{theorem}\label{renewal}
Let $\{X_i\} $ be a sequence of independent identically distributed real valued random variables with $\P[X_i>0]=1$ and $\E[X_i]=m<\infty$.
Let the distribution of $X_i$ be absolutely continuous with respect to the Lebesgue measure. Let $a\in \mathbb R$ be an arbitrary constant and  $S_n=X_1+\cdots+X_n$ with $S_0=0$.  

\begin{itemize}
\item Then the sequence $a+S_n$ for $n\ge 0$ is a point process. It has a limit as a stationary point process $\mathbb Q$ on $\mathbb R$ as $a\to -\infty$. Its restriction to $[0,\infty)$ can be realized as the point process $\{Y+S_n: n\ge 0\}$ where $Y$ is an independent random variable with density  $\frac{1}{m}\P[X_i\ge x]$. 
\item If we have two sets of mutually independent random variables $\{X_i\},\{Y_i\}$ with $\E[X_i]=m_1$ and $E[Y_i]=m_2$ that alternate, i.e
$$
S_{2n}=a+X_1+Y_1+\cdots+X_n+Y_n,
$$ 
and 
$$
S_{2n+1}=a+X_1+Y_1+\cdots+X_n+Y_n+X_{n+1},
$$ 
then again there is a stationary limit for the point process $S_n$ as $a\to -\infty$. Moreover 
$$
\lim_{a\to -\infty} \P\big[\cup_j \{S_{2j}\le x\le S_{2j+1}\}\big]=\frac{m_1}{m_1+m_2},
$$
and
$$
\lim_{a\to -\infty} \P\big[\cup_j \{S_{2j-1}\le x\le S_{2j}\}\big]=\frac{m_2}{m_1+m_2}.
$$
\item There is also an ergodic theorem. We think of $\mathbb R$ as  being covered by intervals of one type or the other
of random lengths and $f(s)=1$ if $s$ is covered by an $X$, i.e $S_{2n}\le s\le S_{2n+1}$ for some $n$ and $0$ otherwise. Then
for any $a\le0$
$$
\lim_{T\to\infty}\frac{1}{T}\int_0^T f(s)ds=\frac{m_1}{m_1+m_2},
$$
with probability $1$. 
\end{itemize}
\end{theorem}
There is a modified version of the above renewal theorem that is relevant to us. 
Recall \eqref{eq2-thm2.5} and also from \eqref{eq-Pi-hat} and \eqref{eq-mu-hat} 
the tilted measures $\widehat\Pi_\alpha$ on active periods $(\xi,\bar u)$ and $\widehat\mu_\alpha$ on dormant periods $\xi^\prime$.

Then $\widehat\Pi_\alpha$ on active clusters    
 provides the  distribution of $|{\mathcal J}(\xi)|$ as well as  the conditional distribution $\nu(X,d\xi)$ of $\xi$ given $|{\mathcal J}(\xi)|=X$, 
 while $\widehat\mu_\alpha$ provides  the distribution of the length $|\xi^\prime|$ of a dormant period. We have the stationary version of the renewal process with alternating active and dormant intervals. Associated with each active period we have the conditional distribution $\nu(X,d\xi)$ of $\xi\in \mathcal Y$ of the birth and death history of the process during the period given its duration and the conditional distribution $\beta(\xi, d{\bar u})$ on $(0,\infty)^{n(\xi)}$ given by the density {$\frac{1}{{{\bf F}}(\xi)}\big(\frac{2}{\pi}\big)^{n(\xi)/2}{\Phi}(\xi, {\bar u}) \d\bar u$}. Like before, we then have convergence as the starting time $a\to -\infty$ to the stationary measure $\widehat{\mathbb Q}_\alpha$ that can be viewed as the distribution of a stationary renewal process with alternating active and dormant intervals and random variables $(\xi, \bar u)$ associated with each active interval with distributions given by the conditionals $\nu( X, d\xi)$ and $\beta(\xi,{\bar u})$. In our context, the ergodic theorem in Theorem \ref{renewal} also translates as follows: $X_n$
 and $Y_n$ are active and dormant intervals with $(\xi_n, u_n)$ associated with $X_n=|{\mathcal J}(\xi_n)|$. With $S_n=a+X_1+Y_1+\cdots+X_n+Y_n$, for any $a\le 0$,
 
\begin{equation}\label{ren-lim}
\begin{aligned}
\lim_{T\to\infty} \frac{1}{T}\sum_{n\atop 0\le S_n\le T} g(\xi_n,{\bar u}_n,Y_n)=\frac{\E[g(\xi,{\bar u,Y})]}{\E[X+Y]}
&=\frac{\E\big[\int g(\xi,{\bar u},Y)\widehat\Pi_\alpha(\d\xi)\beta(\xi, \d{\bar u})\big]}{\E[X]+\E[Y]} \\
&=\frac{\int g(\xi, {\bar u}, s)\widehat\Pi_\alpha (\d\xi)\beta(\xi,\d{\bar u})\widehat\mu_\alpha(ds)}{\E[X]+\E[Y]},
\end{aligned}
\end{equation}
with probability $1$.

%The following result provides the law of large numbers on a single active cluster $(\xi,\bar u)$. 

\begin{lemma}\label{lemma-lln-single}
Let $\Pi_{\alpha,T}$ be the law of the law of birth death process on a $(\xi,\bar u)$ defined in Section \ref{sec-Q-alpha-T}. Then the normalization  constant
$$
{{\widehat Z(\alpha, T)=\E^{\Pi_\alpha,T}\big[\e^{-\lambda |{\mathcal J}(\xi)|}{\bf F}(\xi) \big]}}
$$
converges to
$$
\widehat Z(\alpha)=\E^{\Pi_\alpha}\big[\e^{-\lambda  |{\mathcal J}(\xi)|}{\bf F}(\xi) \big]=\frac{\lambda+\alpha}{\alpha},
$$
and the corresponding tilted measures $\widehat\Pi_{\alpha,T}(\d\xi\,\,\d\bar u)= \frac 1 {\widehat Z_{\alpha,T}} \,\, \e^{-\lambda |{\mathcal J}(\xi)|} \,\, [\Phi(\xi,\bar u)\,\,\d\bar u]\,\, \Pi_{\alpha,T} (\d\xi)$ converge in variation to $\widehat\Pi_{\alpha}$ on
the space $\mathcal Y$.
\end{lemma}
\begin{proof}
Since $\Pi_{\alpha,T}$ converges in variation on $\mathcal Y$ to $ \Pi_{\alpha}$ as $T\to\infty$, the requisite convergence is a question of uniform integrability of $\e^{-\lambda|{\mathcal J}(\xi)|}{\bf F}(\xi)$ with respect to $\Pi_{\alpha,T}$ as $T\to\infty$. Let $\widehat\Theta_\alpha$ be 
the renewal process on $[0,\infty)$ starting with $\widehat \Pi_\alpha$ and alternating with $\widehat\mu_\alpha$. 
$\overline\Theta_{\alpha,T}$ is defined as the restriction of $\widehat\Theta_\alpha$ to the event $E_T=\{N(T)=0\}$$^{\dagger}$\footnotetext{$\dagger$ At any given time $t$, $N(t)$ is the current population size and 
the requirement $N(T)=0$
characterizes the dormant period and is equivalent to saying that there is no interval $[s,t]$ in our realization of the Poisson process that contains the point $T$.} 
normalized by  $q(T)=\widehat\Theta_\alpha(E_T)$. As $T\to\infty $, by Theorem \ref{renewal}, $\lim_{T\to\infty}q(T)=q>0$
exists and $q=\widehat{\mathbb Q}_\alpha (N(0)=0)$, where $\widehat{\mathbb Q}_\alpha$ is the stationary version of $\widehat\Theta_\alpha$. Since $q(T)$ is bounded below and $\e^{-\lambda|{\mathcal J}(\xi)|}{\bf F}(\xi)$ is integrable with respect to $\widehat\Theta_\alpha$, it is uniformly integrable with respect to $\overline\Theta_{\alpha,T}$ and hence with respect to $\widehat\Pi_{\alpha,T}$,     which is the  restriction  of  $\widehat\Theta_{\alpha,T}$ to the first cluster.

\end{proof}
\begin{remark}
The distribution of $\Pi_{\alpha, T}$ as well as that of  $\widehat\Pi_{\alpha,T}$ depend on the starting time of the cluster and $T$ has to be interpreted as the time remaining or $T-\sigma^\ast$, where $\sigma^\ast$ is the starting time of the cluster.
\end{remark}

Recall from the proof of Lemma \ref{lemma-lln-single} that $\widehat\Theta_\alpha$ is
the renewal process on $[0,\infty)$ starting with $\widehat \Pi_\alpha$ and alternating with $\widehat\mu_\alpha$, $\widehat{\mathbb Q}_\alpha$ is its stationary version, and $E_T$ is the event $N(T)=0$.
\begin{lemma}
If we set $\widehat\rho (\alpha, T)=\widehat\Theta_\alpha (E_T)$, %and 
%$$
%\overline\Theta_{\alpha,T}(A)=\frac{1}{\widehat\rho(\alpha, T)}\widehat\Theta_\alpha (A\cap E_T),
%$$
the limit
$$
\lim_{T\to\infty}\widehat\rho(\alpha, T)%= \widehat\Theta_{\alpha}[E_T]
=\widehat\rho(\alpha)>0
$$
exists. In particular
$$
\lim_{T_1,T_2 \to\infty} \frac{\widehat\rho(\alpha,T_1)}{\widehat\rho(\alpha, T_2)}=1.
$$
\end{lemma}
\begin{proof}
We start the renewal process at time $0$. Our distributions are all easily seen to be absolutely continuous. We can apply the 
renewal theorem and conclude that
$$
\lim_{T\to\infty} \widehat\rho(\alpha,T)=\lim_{T\to\infty} \widehat\Theta_\alpha (E_T)=\widehat{\mathbb Q}_\alpha(E_0)=\widehat\rho(\alpha)>0
$$

\end{proof}

We are ready to state the law of large numbers for $\widehat\Theta_{\alpha,T}$.

\begin{theorem}[Asymptotic behavior of mixing measure $\widehat\Theta_{\alpha,T}$ as $T\to\infty$.]\label{thm6}
Let $\widehat\Theta_{\alpha,T}$ be the mixing measure for the Polaron measure defined in \eqref{eq-Theta-hat-T}, while $\widehat\Theta_\alpha$ is
the renewal process on $[0,\infty)$ starting with $\widehat \Pi_\alpha$ and alternating with $\widehat\mu_\alpha$, while $\widehat{ \mathbb Q}_\alpha$ is the stationary version of the process $\widehat \Theta_\alpha$. 
Let $[T_1,T_2]\subset [0,T]$ be an interval such that
$T_1\to\infty$ and $T-T_2\to \infty$. Then the total variation $\|\widehat {\mathbb Q}_\alpha-\widehat \Theta_{\alpha, T}\|$ on the $\sigma$-field of all excursions in $[T_1, T_2]$ tends to $0$.

\end{theorem}
\begin{proof}
The proof is carried out in two steps. First let us compare $\widehat\Theta_{\alpha, T}$ with $\widehat\Theta_\alpha$ on the $\sigma$-field $\mathcal F_{\tau(T_2)}$ where $\tau(s)=\inf \{t: t\ge s, N(t)=0\}$ is the  first time after $s$ that the population size is $0$. If $T-T_2$ is large then we can find $C(T)\to \infty$ with $T$ such that  the event $\tau(T_2)\ge T-C(T)$ has small probability under both $\widehat \Theta_{\alpha}$ and $\widehat \Theta_{\alpha,T}$.  We also have
$$
\widehat\Theta_{\alpha,T}(A)=[\widehat\Theta_\alpha (E_T)]^{-1} \widehat\Theta_\alpha (A\cap E_T)
$$ 
Therefore the Radon-Nikodym derivative $\frac{d\widehat\Theta_{\alpha,T}}{d\widehat\Theta_\alpha}=\frac{\widehat \rho(\alpha, T-\tau(T_2))}{\widehat \rho(\alpha, T)}$ on $\tau(T_2) < T$  is nearly $1$ with high probability. Since $\widehat\rho$ has a nonzero limit 
that is bounded by $1$, the ratio is never large, making $\|\widehat \Theta_{\alpha,T}-\widehat\Theta_\alpha\|$ small on $[0, T_2]$ if $T-T_2\gg 1$. On the other hand according to standard renewal theorem $\|\widehat{\mathbb Q}_\alpha-\widehat\Theta_\alpha\|$ is small on $[T_1,\infty]$ if $T_1\gg1$. 
\end{proof}

%\begin{remark}\label{rmk-lambda-alpha}
%Note that Theorem \ref{thm6} determines  the choice of $\lambda(\alpha)=\lim_{T\to\infty}\frac 1T\log Z_{\alpha,T}=g(\alpha)$ satisfying \eqref{eq2-thm2.5}. %where we need to pick $\lambda(\alpha)=$.
%\end{remark}

\subsection{Proof of Theorem \ref{thm4}, Theorem \ref{thm5} and Theorem \ref{thm4.5}.}\label{sec-proofs-sec-4}

We start with the proof of Theorem \ref{thm4}.

\noindent{\bf{Proof of Theorem \ref{thm4}.}} Theorem \ref{thm4} is a direct consequence of the Gaussian representation proved in Theorem \ref{thm1} and the law of the large numbers for the mixing measure $\widehat\Theta_{\alpha,T}$ provided by Theorem \ref{thm6}.
\qed

%In order to prove Theorem \ref{thm5}, we again need \eqref{ren-lim}.

\bigskip
We will now prove the central limit theorem. 

\noindent{\bf{Proof of Theorem \ref{thm5}.}}
By Theorem \ref{thm1}, the finite-volume polaron measure $\widehat\P_{\alpha,T}$ is a superposition of Gaussian measures with values in $\R^3$ that are rotationally symmetric and  indexed by  $(\widehat{\xi},{\widehat u})\in  \mathcal{\widehat Y}$, while by Theorem \ref{thm4}, the infinite-volume polaron measure $\widehat\P_\alpha=\lim_{T\to\infty}\widehat\P_{\alpha,T}$ is also a similar superposition of Gaussian measures (recall \eqref{eq-Polaron-limit}). 
The mixing measure corresponding to the Gaussian representation of $\widehat\P_{\alpha,T}$ is ${\widehat \Theta}_{\alpha, T}$ on  $\widehat{\mathcal Y}$, 
while the same for its infinite-volume counterpart being $\widehat\Theta_\alpha$ constructed in Theorem \ref{thm6}. 
Let us first prove the CLT for the distribution $\widehat\nu_{\alpha,T}$ of  $y(T)=\frac{1}{\sqrt{2T}}[x(T)-x(-T)]$ under $\widehat\P_{\alpha,T}$ (by the previous consideration, the CLT for 
for the distribution $\widehat\mu_{\alpha,T}$ of  $y(T)=\frac{1}{\sqrt{2T}}[x(T)-x(-T)]$ under $\widehat\P_{\alpha}$ can be shown exactly in the same manner).

Then $\widehat\nu_{\alpha,T}$ is a rotationally symmetric,  mean zero Gaussian with covariance $Z$ times identity where $Z=Z(\widehat{\xi},{\widehat u})$ is random.  The central   limit theorem for $y(T)$ can be established by proving $\lim_{T\to\infty} \E^{\widehat \Theta_{\alpha, T}}[|Z-c|]=0$ for some constant $c$. The interval  $[-T,T]$
is divided into $2k+1$ intervals, $k+1$ of them dormant and $k$ that are active. All the Gaussian measures in the superposition have independent increments over these intervals. In an active interval $\xi$, the variance of the increment in any component is given by

\begin{equation}\label{unibound}
\begin{aligned}
\sigma^2(\xi,\bar u)&=\frac 1{\Phi(\xi,\bar u)} \,\E^\P\bigg[ \exp\bigg\{-\frac 12 \sum_{i=1}^{n(\xi)} u_i^2\,|\omega(t_i)-\omega(s_i)|^2\bigg\}\,\, \big\langle v, (\omega(\sigma^\star)-\omega(0))\big\rangle^2\bigg] \\
&\leq |\mathcal J(\xi)|=\sigma^\star.
\end{aligned}
\end{equation}
where 
$$\Phi(\xi,\bar u)=\E^\P[\exp\{-\frac 12 \sum_{i=1}^n u_i^2 |\omega(t_i)-\omega(s_i)|^2\}]$$
and in a dormant interval $\xi^\prime$,
$$
\sigma^2(\xi^\prime)=|\xi^\prime|.
$$
Then,
$$
Z=\frac{1}{2T}\bigg[\sum_{i=1}^{k+1}\sigma^2(\xi_i^\prime)+ \sum_{i=1}^k\sigma^2(\xi_i^\prime)\bigg].
$$
The uniform bound  (\ref{unibound}) allows us to replace $[-T,T]$ with a smaller interval $[-T_1,T_1]$ with $T-T_1$ which is  large but $o(T)$ for large $T$. 

We can then  apply Theorem \ref{thm6}  to  replace $\widehat{\Theta}_{\alpha,T}$ with  ${\mathbb Q}_\alpha$ and applying Theorem \ref{thm4},  Theorem \ref{renewal},  
 the ergodic theorem stated in \eqref{ren-lim} with $g(\xi,\bar u,s)=\sigma^2(\xi, \bar u)+s$, and noting that  the variance in a dormant period is $(\lambda+\alpha)^{-1}$, we conclude the result.
\qed

We will now prove Theorem \ref{thm4.5} for which we will need the following estimate. 

\begin{lemma}\label{lemma-2-thm4.5}
Let  $f(A)= \widehat{\mathbb Q}_\alpha\bigg\{[-A,A]\subset \mathcal J(\xi)\colon \xi \,\,\mbox{is a single active period}\bigg\}$, where $\widehat{\mathbb Q}_\alpha$ is the stationary version of the tilted distribution $\widehat\Theta_\alpha$ (recall Theorem \ref{thm6}). 
Then there is $\alpha_0\in (0,\infty)$ such that for $\alpha\in (0,\alpha_0)$, there are constants $c>0$ and $ C<\infty$ such that $f(A)\le Ce^{-cA}$.
\end{lemma}
\begin{proof}
If $-A$ is in a dormant period the required estimate is trivially true. Assume $-A$ is in an active interval $\xi$. Then the distribution of the time to the beginning  of the next active period is the tail probability of  the distribution of the sum of the durations of an active and dormant period. So it has exponential decay since by Theorem \ref{thm2}, $\E^{{\widehat\Pi}_{\alpha}}[ \exp \big\{a\,\,|\mathcal J(\xi)|\big\}]<\infty$ for small enough $\alpha$. 
\end{proof}
%We finally turn to the proof of Theorem \ref{thm4.5}.

\noindent{\bf{Proof of Theorem \ref{thm4.5}.}} 
Let us consider the time between the start of two successive active periods. This is a random variable which has  a distribution with an exponential decay by Lemma \ref{lemma-2-thm4.5}, and 
 starting from $-A$ the probability that no  renewal  takes place before $A$ is at most $Ce^{-2cA}$ for some $c>0, C<\infty$.
Clearly, if there is a renewal, then ${\widehat\P}_\alpha$ on $(-\infty,-A)$ and $(A,\infty)$ are independent.
\qed

\section{Outlook: The strong coupling limit $\alpha\to\infty$, the mean-field Polaron and the Pekar process.}\label{sec-6}

We first remark that, for any fixed $\alpha>0$ and $T>0$, we have the distributional identity 
\begin{equation}\label{alphaeps}
\int_{-T}^T\int_{-T}^T \d s \d t \frac{\alpha \e^{-|t-s|}}{|\omega(s)-\omega(t)|} \stackrel{\mathrm{(d)}}{=} \int_{-\alpha^2 T}^{\alpha^2T}\int_{-\alpha^2T}^{\alpha^2T} \d s \d t \frac{\alpha^{-2} \e^{-\alpha^{-2}|t-s|}}{|\omega(s)-\omega(t)|}.
\end{equation}
If we are interested in the strong-coupling limit $\alpha\to\infty$ of infinite volume measure $\widehat\P_\alpha$, 
a slight reformulation of the Polaron measure has to be considered, which is given by a  Kac-interaction of the form
\begin{equation}\label{6.2}
\widehat{\mathbb P}_{\eps,T}(\d\omega)\stackrel{\ssup{\mathrm{def}}}= \frac 1 {Z_{\eps,T}} \, \exp\bigg\{\int_{-T}^T\int_{-T}^T \, \d s \d t \frac{\eps\e^{-\eps|t-s|}}{|\omega(s)-\omega(t)|}\bigg\} \, \mathbb P(\d\omega).
\end{equation}
Note that given the distributional identity (6.1), the coupling parameter $\alpha$ is related to {\it Kac-parameter} $\eps$ via the relation $\eps=\alpha^{-2}$. Note that in this context, we again have a Poisson point process with intensity $\eps\e^{-\eps (t-s)}$ for $s<t$ and the corresponding law of the birth and death process in one single cluster has birth rate $\eps>0$ and death rate $1$. %The validity of latter statement follows

Given the representation \eqref{6.2}, the {\it strong coupling limit} $\alpha\to\infty$ now translates to the {\it Kac limit} $\eps\to 0$ for $\widehat\P_\eps$. To describe this limit, it is useful to go back to the investigation 
of the strong-coupling ground state energy carried out in \cite{DV83}. Recall from \eqref{eq-DV-0} that 
$\lim_{T\to\infty}\frac 1 T\log Z_{\eps,T}=g(\eps)=\sup_{\mathbb Q}\big[\E^{\mathbb Q}\bigg\{\int_0^\infty \frac{\eps\e^{-\eps r} \, \d r}{|\omega(r)-\omega(0)|}\bigg\}- H(\mathbb Q)\big]$
with the supremum being taken over all stationary processes $\mathbb Q$ in $\R^3$. 
Indeed, the value of $g(\eps)$ is not altered if we take the same supremum 
over {\it processes with stationary increments in $\R^3$}, and as before $\lim_{\eps\to 0}\lambda(\eps)=\lim_{\eps\to 0} g(\eps)= g_0$ with $g_0$ defined in \eqref{eq-DV}. 
From a statistical mechanical point of view,  the latter result basically implies that in the strong-coupling regime, at least the partition function $Z_{\eps,T}$ behaves in leading order  like the  partition function of the so-called {\it{mean-field Polaron}}.
More precisely, with the Pekar variational formula defined in \eqref{eq-DV}, we have 
$$
g_0=\lim_{\eps\to 0}\lim_{\lim_{T\to\infty}} \frac 1 T \log Z_{\eps,T}=\lim_{T\to\infty} \frac 1 T\log Z_T^{\ssup{\mathrm{mf}}} 
$$
where $Z_T^{\ssup{\mathrm{mf}}}$ is the partition function for the {\it{mean-field Polaron measure}}
$$
\widehat\P_{T}^{\ssup{\mathrm{mf}}}(\d \omega)= \frac 1 {Z_{T}^{\ssup{\mathrm{mf}}}} \exp\bigg\{\frac{1}T \int_{0}^T\int_0^T  \frac{\d t  \d s}{|\omega_t- \omega_s|}\bigg\} \,\, \P_0(\d \omega),
$$
with $\P_0$ denoting the law of Brownian motion starting the origin in $\R^3$. That is, the so-called ``mean-field approximation" for the Polaron in strong coupling is valid, at least for the (leading term) of the partition function. It is natural to wonder if such approximation continues to remain valid also on the level of the actual path measures when the coupling is large. 

The limiting behavior of the mean-field measures $\widehat\P_T^{\ssup{\mathrm{mf}}}$ as $T\to\infty$ have been  fully analyzed
recently in a series of results (\cite{MV14, KM15, BKM15}), where it is shown that the distribution 
$\widehat\P_T^{\ssup{\mathrm{mf}}}\, L_T^{-1}$ of the Brownian occupation measures $L_T=\frac 1 T\int_0^T\delta_{W_s} \,\d s $ under $\widehat\P_T{\ssup{\mathrm{mf}}}$ converges to the distribution 
of a random translation $[\psi_0^2\star \delta_X]\,\d z$ of $\psi_0^2 \, \d z$, with the random shift $X$ having a density $\psi_0 /\int \psi_0$. 
 Furthermore, it was also shown in \cite{BKM15} that the mean-field measures $\widehat\P_T{\ssup{\mathrm{mf}}}$ themselves converge, as $T\to\infty$ towards
 a spatially inhomogeneous mixture of the stationary process driven by the SDE 
$\d X_t= \d W_t+ (\frac{\nabla \psi_0}{\psi_0})(W_t)\,\d t$
 with the spatial mixture being taken w.r.t. the weight $\psi_0/\int \psi_0$. This result consequently led to a rigorous construction of the {\it{Pekar process}}, a stationary diffusion 
 process with generator $\frac 12 \Delta+ (\nabla \psi_x/\psi_x)\cdot\nabla$ with $\psi_x^2=\psi_0^2\star \delta_x$, whose heuristic definition was 
 set forth by Spohn in \cite{Sp87}. Note that, while the Pekar process is not uniquely defined, its increment process is uniquely determined by {\it any maximizer} $\psi$ of the variational formula $g_0$. if $\widehat{\mathbb Q}_\psi$ denotes the stationary version of the increments of the Pekar process, recently we have shown (\cite{MV18})
 $$
 \lim_{\eps\to 0} \widehat\P_\eps= \widehat{\mathbb Q}_\psi
 $$
 justifying the ``mean-field approximation" of the strong coupling Polaron even for path measures, which was also conjectured by Spohn in \cite{Sp87}.

\section{Appendix}\label{sec-appendix}

In this appendix we collect some estimates w.r.t. birth and death processes. %While These estimates have not been used for our main results, but it might be of independent interest.   
Recall that $\Pi_\alpha$ denotes the law of the birth-death process on a single cluster starting with one individual, with birth rate $\alpha>0$ and death rate $1$, while $\mu_\alpha$ is the exponential distribution on any gap with parameter $\alpha$.

\begin{theorem}\label{thm2.6old}
There exists $\alpha_0\in (0,\infty)$ such that if $\alpha\in(0,\alpha_0)$ then for some $\lambda=\lambda(\alpha)\in (0,\infty)$, 
\begin{equation}\label{eq-thm2.6}
\begin{aligned}
&q_\lambda(\alpha)=\E^{\Pi_\alpha\otimes\mu_\alpha}\bigg[\e^{-\lambda(\alpha)[|\mathcal J(\xi)+ |\xi^\prime|]}\,\, \mathbf F(\xi)\bigg]=1,
\\
&L_\lambda(\alpha)=\E^{\Pi_\alpha\otimes\mu_\alpha}\bigg[ \big[|\mathcal J(\xi)|+|\xi^\prime|\big]\,\,\e^{-\lambda(\alpha)[|\mathcal J(\xi)+ |\xi^\prime|]}\,\, \mathbf F(\xi)\bigg]<\infty. 
\end{aligned}
\end{equation}
for any active period $\xi$ and dormant period $\xi^\prime$.
\end{theorem}

Recall  that in Theorem \ref{thm2.6} the above statement was shown to be true for any $\alpha>0$ and for $\lambda=\lambda(\alpha)=g(\alpha)-\alpha$. 
We will now give a different proof of Theorem \ref{thm2.6old} which are based on explicit estimates on the underlying determinant. These estimates are sufficient
 to conclude the above statement when $\alpha>0$ is sufficiently small, but do not seem to be strong enough for our purposes if $\alpha>0$ is large. 
 First we will need

 \begin{theorem}\label{thm2}
There exists $\alpha_0\in (0,\infty)$ so that for $\alpha\in(0,\alpha_0)$, 
\begin{equation}\label{eq-1-thm2}
\begin{aligned}
& \E^{\Pi_\alpha}\big[\mathbf F(\xi)\big] <\infty, \qquad
\E^{\Pi_\alpha}\big[|\mathcal J(\xi)|\,\,\mathbf F(\xi)\big]  <\infty, \qquad\mbox{and  } \\
&\E^{\Pi_\alpha}\big[\e^{a|\mathcal J(\xi)|}\,\mathbf F(\xi)\big]  <\infty, \qquad\mbox{for some }a>0.
\end{aligned}
\end{equation}
   
\end{theorem}

\begin{remark}
Theorem \ref{thm2} is not true if $\alpha$ is large: In Lemma \ref{alpha-large} it is shown that there is $\alpha^\star$ such that if $\alpha>\alpha^\star$, then $\E^{\Pi_\alpha}\big[\mathbf F(\xi)\big]=\infty$.
\end{remark} 
Assuming Theorem \ref{thm2} let us first conclude 

{\bf Proof of Theorem \ref{thm2.6old}:} By Theorem \ref{thm2}, for $\alpha<\alpha_0$, $q_0(\alpha)=\E^{\Pi_\alpha\otimes\mu_\alpha} [\mathbf F(\xi)]<\infty$, while by Lemma \ref{lemma-sqrt2}, $q_0(\alpha)\geq \sqrt 2\geq 1$. 
The function $\lambda\mapsto q_\lambda $ is continuous and  monotone decreasing in $\lambda\in (0,\infty)$ and $q_\lambda\downarrow 0$ as $\lambda\to\infty$. Then we can find $\lambda=\lambda(\alpha)$ such that 
$q_\lambda(\alpha)=1$. The finiteness of $L_\lambda$ follows from the observation that $q_\lambda<\infty$ for a slightly lower value of  $\lambda$.
\qed

We now owe the reader the proof of Theorem \ref{thm2} which is carried out in few steps. The first step is to prove the following upper bound on the total mass $\Phi(\xi, \bar u)$.

\begin{lemma}\label{lemma-2-pf-thm2}
We have an upper bound 
$$
\Phi(\xi,\bar u) \leq \prod_{i=1}^{n(\xi)} (1+ u_i^2 \delta_i)^{-3/2}.
$$
where $\delta_i$ is defined in \eqref{eq-delta-def}. \end{lemma}

The proof of the Lemma \ref{lemma-2-pf-thm2} depends on the following estimate.
\begin{lemma}\label{lemma-1-pf-thm2}
Let $M=(m_{ij})$ be any symmetric positive definite $n\times n$ matrix with real-valued entries. Then 
\begin{equation}\label{eq2-lemma1-pf-thm2}
\mathrm{Det}(I+M)\geq \prod_{i=1}^n (1+\gamma_i^2)
\end{equation}
where
$$
\gamma_i^2=\E\bigg[ \big(X_i -\E(X_i| X_1,\dots, X_{i-1})\big)^2\bigg]
$$
and $(X_1,\dots,X_n)$ is a mean zero Gaussian vector $(X_1,\dots,X_n)$ with $m_{ij}= \E(X_i X_j)$.
\end{lemma}
\begin{proof}
Since any symmetric positive definite matrix $M$ is the covariance matrix of a mean zero Gaussian vector $(X_1,\dots,X_n)$, it is well-known that
\begin{equation}\label{eq1-lemma1-pf-thm2}
\mathrm{Det}(M)= \prod_{i=1}^n \gamma_i^2=\prod_{i=1}^n \E\bigg[ \big(X_i -\E(X_i| X_1,\dots, X_{i-1})\big)^2\bigg]
\end{equation}
The above identity will imply the desired the lower bound \eqref{eq2-lemma1-pf-thm2} as follows. Let $Y_1,\dots, Y_n$ be an independent set of standard Gaussian random variables which are also
independent of $\{X_1,\dots,X_n\}$. Let us denote by $Z_i=X_i+Y_i$. Then, by the identity we just proved above, 
$$
\det(I+M)= \prod_{i=1}^n \beta_i^2 
$$
where, for any $i=1,\dots,n$
$$
\begin{aligned}
\beta_i^2 &= \E\big[ \big(Z_i -\E(Z_i| Z_1,\dots, Z_{i-1})\big)^2\big] \\
&\geq \E\big[ \big(Z_i -\E(Z_i| X_1,Y_1,\dots,X_{i-1},Y_{i-1})\big)^2\big] \\
&=\E\big[ \big(X_i -\E(X_i| X_1,\dots, X_{i-1})\big)^2\big]+ \E\big[ Y_i^2\big] \\
&=\gamma_i^2+1.
\end{aligned}
$$

This concludes the proof of the lower bound \eqref{eq2-lemma1-pf-thm2} and that of Lemma \ref{lemma-1-pf-thm2}.
\end{proof}

We will now complete the proof of Lemma \ref{lemma-2-pf-thm2}.

\noindent{\bf{Proof of Lemma \ref{lemma-2-pf-thm2}}.}
Recall that 
$$
\Phi(\xi, \bar u)=\E^{\P}\bigg[\exp\bigg\{-\frac{1}{2}\sum_{i=1}^{n(\xi)}  u_i^2\big|\omega(t_i)-\omega(s_i)\big|^2\bigg\}\bigg]
$$

Let 
$$
X_i=u_i\big(\omega(t_i)-\omega(s_i)\big) \qquad i=1,\dots, n(\xi)
$$
denote the rescaled one-dimensional increments. Then $\{X_i\}_{i=1}^{n(\xi)}$ is a mean $0$ Gaussian vector with covariance matrix $C=(C_{ik})$ given by 
 $$
 \E^\P[X_iX_k]= C_{i,k} =u_iu_k |J_i\cap J_k|.
 $$
The  expectation is given by
$$
\E^{\P}\bigg[\exp\bigg\{-\frac{1}{2}\sum_{i=1}^n X_i^2\bigg\}\bigg] = [\det(C^{-1})]^{1/2}\,\, [\det(I+C^{-1})]^{-1/2}= [\det(I+C)]^{-1/2}
$$
and the total mass is given by
$$
\Phi(\xi,\bar u)= [\det(I+C)]^{-3/2}.
$$
For notational convenience, we will write $J_i=[s_i,t_i]$ so that $\mathcal J(\xi)=[0,\sigma^\star]=\cup_{i=1}^{n(\xi)} J_i$. 
We have    
$$
0=\sigma_0<\cdots<\sigma_{2n(\xi)-1}=\sigma^\ast
$$
 that divides the interval $[0,\sigma^\ast]$ into $(2n(\xi)-1)$
intervals $U_r=[\sigma_{r-1},\sigma_r]$, $1\le r\le 2n(\xi)-1$. Each $J_i$ is the union of a set of $U_r$. We label $\{J_i\}$ so that $t_1<t_2<\cdots<t_n=\sigma^\ast$ and similarly order the disjoint intervals  $U_r$.   

{ Let $\theta_i$ be the increment $\omega(t_i)-\omega(s_i)$ over $J_i$,  while $\{\eta_i\}$ are independent Gaussians with mean $0$ and variance $1$.} We set
$$
\zeta_i=u_i\theta_i+\eta_i.
$$
 Let us fix an $i$ and with $t_i=\sigma_{r(i)}$, $U_r$ is the interval $[\sigma_{r(i)-1},\sigma_{r(i)}]$.    If
\begin{equation}
q_i=\inf_{a_1,\ldots,a_{i-1}} \E\bigg[ \big(\zeta_i-(a_1\zeta_1+\cdots+a_{i-1}\zeta_{i-1})\big)^2\bigg]
\end{equation}
where the expectation is with respect to both the Brownian increments $\{\theta_i\}$ and $\{\eta_i\}$, 
then, as in the proof of Lemma \ref{lemma-1-pf-thm2}, 
$$
\det(I+C)= \prod_{i=1}^{n(\xi)} q_i
$$
%q_2\cdots q_n$.  
Now we will get a lower bound on $q_i$. With  $Z_r$ being  the increment of $\omega(\sigma_r)-\omega(\sigma_{r-1})$ over $U_r$, we again use Lemma \ref{lemma-1-pf-thm2} to obtain
$$
\begin{aligned}
q_i&\ge \inf_{b_1,\ldots,b_{r(i)-1}\atop k_1,\ldots, k_{i-1}} \E\bigg[ \bigg\{u_i\theta_i+\eta_i-(b_1Z_1+\cdots+b_{r(i)-1}Z_{r(i)-1}+k_1\eta_1+\cdots +k_{i-1}\eta_{i-1})\bigg\}^2\bigg]\\
&=(1+u_i^2\delta_i)
\end{aligned}
$$
The proof of Lemma \ref{lemma-2-pf-thm2} is therefore finished.
\qed
 
 We will now prove Theorem \ref{thm2}.
 
 \noindent{\bf{Proof of Theorem \ref{thm2}.}} From the upper bound coming from Lemma \ref{lemma-2-pf-thm2}, it follows that, 
 \begin{equation}\label{eq-claim0-thm2}
\mathbf F(\xi)=  \int_{(0,\infty)^{n(\xi)}}  \bigg(\sqrt{\frac 2\pi}\bigg)^{n(\xi)}\Phi(\xi,\bar u)\ \Pi du_i  \leq \prod_{i=1}^{n(\xi)} \bigg(\sqrt{\frac 2\pi} \,\frac 1 {\sqrt\delta_i}\bigg),
 \end{equation}
 and it suffices to estimate 
 $$
 M:= \E^{\Pi_\alpha}\bigg[\prod_{i=1}^{n(\xi)}\bigg(\sqrt{\frac 2\pi} \,\frac 1 {\sqrt\delta_i}\bigg) \bigg].
 $$
 We condition with respect to the Markov chain of  successive population sizes  $\{X_r\}$, $ 0\le r\le \ell$. $X_0=1, X_\ell=0$, and as remarked before, conditionally  $\{\delta_i\}$ are  independent and  exponentially distributed with rate $n_r+\alpha$, if $\delta_i=\sigma_{r+1}-\sigma_r$. Then 
 \begin{equation}\label{eq-claim0.5-thm2}
  \E^{\Pi_\alpha}\bigg[\bigg(\sqrt{\frac 2\pi} \,\frac 1 {\sqrt\delta_i}\bigg)\bigg| X_r=n_r\bigg]= \sqrt 2\, \sqrt{n_r+\alpha},
 \end{equation}
yielding
  $$
 M\leq \E^{\Pi_\alpha}\bigg[\e^{\sum_{r=1}^{\ell} V(X_r)}\bigg]
 $$ 
 where 
 $$
 V(n)= c_1+ \frac 12 \log(n+\alpha)
 $$
 for a suitable choice of $c_1>0$. For our purposes, it suffices to show that, for any constant $c_1$, there exists $\alpha_0>0$ such that 
\begin{equation}\label{eq-claim-1-thm2}
 \E^{\Pi_\alpha}\bigg[\e^{\sum_{r=0}^{\ell-1} V(X_r)}\bigg| X_0=1\bigg] <\infty,
 \end{equation}
 for $\alpha<\alpha_0$. Consider the function $u(n)=C^n (n!)^{\frac12}$ for some some $C>0$. 
 Since the transition probabilities of the Markov chain $(X_r)$ are given by
 $\pi_{n,n+1}=\frac{\alpha}{n+\alpha}$ and $\pi_{n,n-1}=\frac{n}{n+\alpha}$, with $(\Pi u)(n)=\sum \pi(n,n^\prime) u(n^\prime)$, we have for $n\ge 1$
 $$
 \frac{(\Pi u)(n)}{u(n)}=\frac{\alpha}{\alpha+n}[C\sqrt {n+1}]+\frac{n}{\alpha+n}\frac {1}{C\sqrt{n}}\le \frac{2\alpha C}{\sqrt {n+\alpha}}+\frac{1}{C\sqrt{n+\alpha}}
 $$
 and
\begin{equation}\label{eq-claim-2-thm2}
\log \bigg(\frac{u(n)}{(\Pi u)(n)}\bigg)\ge \frac{1}{2}\log (n+\alpha)-\log(2\alpha C+C^{-1})\ge  -\frac{1}{2}\log\alpha-\log 3+\frac{1}{2}\log (n+\alpha)
\end{equation}
if we choose $C=\frac{1}{\sqrt{\alpha}}$.  If we choose $\alpha<\alpha_0(c_1)$ we can have 
$$
\log \bigg(\frac{u(n)}{(\Pi u)(n)}\bigg)\ge c_1+V(n)
$$

Let us denote by $W=\log(u/\Pi u)$. Then if $\mathbb Q^{\ssup x}$ is the law of a Markov chain
$(X_j)_{j\geq 0}$  starting at $x$, then  it follows from successive conditioning and the Markov property that
$$
\E^{\mathbb Q^{\ssup x}} \bigg[\exp\bigg\{\sum_{j=0}^{\ell-1} W(X_j)\bigg)\bigg\} \, (\Pi u)(X_{\ell-1})\bigg] =u(x), 
$$
implying that
$$
\E^{\mathbb Q^{\ssup x}} \bigg[\exp\bigg\{\sum_{j=0}^{\ell-1} W(X_j)\bigg)\bigg\}\bigg] \leq \frac {u(x)}{\inf_{y} u(y)}.
$$

Since for our choice, $u(0)=1$ and $u(1)=C$, the above estimate, combined with \eqref{eq-claim-2-thm2} implies the claim \eqref{eq-claim-1-thm2}.  

We can also obtain
$$
\E^{\Pi_\alpha}\big[|\mathcal J(\xi)|\mathbf F(\xi)\big] \le \E^{\Pi_\alpha}\bigg[\frac{n(\xi) (\sqrt{2/\pi})^{n(\xi)} }{\sqrt{\delta_1\cdots\delta_n}}\bigg]\le C^\prime
$$
by increasing the value of $c_1$  in \eqref{eq-claim-2-thm2} which  will  then yield, for some $a>0$,   bounds on the exponential moments 
$$
\frac{1}{q_\alpha(0)}\E^{\Pi_\alpha}\bigg [ \e^{an(\xi)} \mathbf F(\xi)\bigg]\le C^{\prime\prime}, 
$$
as well as
\begin{equation}\label{eq-length-expmom}
\frac{1}{ q_\alpha(0)}\E^{\Pi_\alpha}\bigg [ \e^{a|\mathcal J(\xi)|} \mathbf F(\xi) \bigg]\le C^{\prime\prime}.
\end{equation}
This concludes the proof of Theorem \ref{thm2}.\qed

 \begin{remark}\label{rmk-thm2}
Using \eqref{eq-claim0-thm2}, the function  
$$
\mathbf F(\xi)= \bigg(\sqrt{\frac 2\pi}\bigg)^{n(\xi)} \int_{(0,\infty)^{n(\xi)}} \d \bar u\,\,\Phi(\xi,\bar u)
$$
is easily seen to be dominated by the function 
\begin{equation}\label{eq-Psi}
\widehat{\mathbf F}(\xi)= (c_2)^{n(\xi)}\prod_{i=1}^{n(\xi)}\bigg(1+\frac{1}{\sqrt {\delta_i}}\bigg)
\end{equation}
for some $c_2>1$, and one can verify by increasing $c_1$ in \eqref{eq-claim-2-thm2}, that there is  a new $\alpha_0$ such that  for $\alpha<\alpha_0$
\begin{equation}
\E^{\Pi_\alpha}[\widehat{\mathbf F}(\xi)]<\infty.
\end{equation}
It is worth noting that the function $\widehat{\mathbf F}(\xi)$ is monotone in the sense that if $\xi^\prime\supset \xi$ then $\widehat{\mathbf F}(\xi^\prime)\ge \widehat{\mathbf F}(\xi)$.
\end{remark}

\medskip

The following lemma shows that if $\alpha>0$ is sufficiently large, then the statement of Theorem \ref{thm2} no longer holds:

\begin{lemma}\label{alpha-large}
 There exists $\alpha^\star<\infty$ such that for $\alpha>\alpha^\star$, 
 $\E^{\Pi_\alpha}[\mathbf F(\xi)]=\infty$. 
 \end{lemma}
 \begin{proof}
 We note that $n(\xi)$ is the total number of births including the one at the start. If any one individual has a life time that is more than $N$, then the active period is at least of duration $N$, and   with birth rate of $\alpha$, the total number of births would be  at least about $\alpha N$. Hence
$$
\Pi_\alpha\big[n\simeq \alpha N\big]\ge \e^{-N}.
$$
which, combined with \eqref{eq-lemma-sqrt2} implies that $\E^{\Pi_\alpha}[\mathbf F(\xi)]=\infty$ if $\alpha$ is large enough.
\end{proof}

Finally, we end with the following result which shows {\it finiteness} of $q_\alpha$ when $\alpha>0$ is small and $\lambda$ is arbitrary:

\begin{theorem}\label{thm2.5}
There exists $\alpha_0\in (0,\infty)$ such that for any $\lambda>0$, 
$$
\begin{aligned}
&\E^{\Pi_\alpha\otimes\mu_\alpha}\bigg[\e^{-\lambda[|\mathcal J(\xi)|+ |\xi^\prime|]}\,\, \mathbf F(\xi)\bigg]<\infty\qquad\mbox{and}, 
\\
&\E^{\Pi_\alpha\otimes\mu_\alpha}\bigg[ \big[|\mathcal J(\xi)|+|\xi^\prime|\big]\,\,\e^{-\lambda[|\mathcal J(\xi)+ |\xi^\prime|]}\,\, \mathbf F(\xi)\bigg]<\infty. 
\end{aligned}
$$
where $|\mathcal J(\xi)|=\sigma^\star$ denotes the total duration of an active period $\xi$.
\end{theorem}

\begin{proof}
Note that, for any $\lambda>0$, and dormant period $\xi^\prime$, since
$$
\E^{\mu_\alpha}\bigg[\e^{-\lambda|\xi^\prime|}\bigg]=\frac \alpha{\alpha+\lambda},
$$
by our previous estimates (recall \eqref{eq-claim0-thm2}) we only need to check that
$$
\frac \alpha{\alpha+\lambda} \E^{\Pi_\alpha}\bigg[\prod_{i=1}^{n(\xi)} \bigg(\e^{-\lambda\delta_i} \,\, \sqrt{\frac 2 \pi} \,\,\frac 1 {\delta_i}\bigg)\bigg]<\infty.
$$
As before (recall \eqref{eq-claim0.5-thm2}), 
$$
\E^{\Pi_\alpha}\bigg[\e^{-\lambda\delta_i} \,\, \sqrt{\frac 2 \pi} \,\,\frac 1 {\delta_i}\bigg | X_r=n_r\bigg] = \sqrt 2 \,\,\frac{n_r+\alpha}{\sqrt{n_r+\alpha+ \lambda(\alpha)}},
$$
and we need to find a function $u$ such that 
\begin{equation}\label{claim-thm2.5}
\begin{aligned}
& \frac{n+\alpha}{\sqrt{n+\alpha+ \lambda}} \bigg[\frac{\Pi u(n)} {u(n)}\bigg] \\
&= \frac{1}{\sqrt{n+\alpha+ \lambda}} \bigg[ \alpha \frac{u(n+1)}{u(n)}+ n \frac{u(n-1)}{u(n)} \bigg]
\leq \frac 1 {\sqrt 2}.
\end{aligned}
\end{equation}
We can again choose $u(n)= c_2^n \,\, (n!)^{1/2}$, for some $c_2=c_2(\alpha)$ so that the left hand side in the last display can be estimated from above by 
$2\alpha c_2 + \frac 1 {c_2}$. If we now set $c_2=\frac 1 {\sqrt\alpha}$, and choose $\alpha\in (0,\alpha_0)$ small enough then we have \eqref{claim-thm2.5}. 
%\begin{equation}\label{eq-lambda}
%\lambda(\alpha) = 3\sqrt 2 \alpha^{\frac 32},
%\end{equation}
%we prove the claim \eqref{claim-thm2.5}. 
\end{proof}

%%      ---------------------------------------------------------------------
%%      ------------------------- APPENDIX (OPTIONAL) -----------------------
%%      ---------------------------------------------------------------------
        
%%      If you have one appendix, uncomment the line \appendix and add
%%      a \section{ *** APPENDIX TITLE ***}. If you have more than
%%      one, uncomment the line \appendices and add a \section{ ***
%%      APPENDIX TITLE ***} command for each appendix title.

%\appendix
%\appendices
%\section{}

%%      Type body of appendix/-ices here.

%%      ---------------------------------------------------------------------
%%      ---------------------------ACKNOWLEDGMENTS (OPTIONAL) ---------------
%%      ---------------------------------------------------------------------

%% ***** UNCOMMENT THE FOLLOWING LINE TO ADD ACKNOWLEDGMENTS.

 \ack The first author would like to thank Erwin Bolthausen, Wolfgang Koenig  and Herbert Spohn  for many inspiring and helpful discussions
 on the Polaron problem. Both authors would like to thank Erwin Bolthausen and Amir Dembo for pointing out an error in Theorem \ref{thm4} in an earlier version of the article. 
 Also Theorem \ref{thm2.6} was earlier stated for $\alpha\in (0,\alpha_0) \cup (\alpha_1,\infty)$, while its proof for $\alpha\in (\alpha_1,\infty)$ in the earlier version 
 had a gap that was pointed out by Volker Betz and Steffen Polzer who we sincerely thank.

%%      ---------------------------------------------------------------------
%%      --------------------------- BIBLIOGRAPHY ----------------------------
%%      ---------------------------------------------------------------------

\frenchspacing
\bibliographystyle{cpam}

\end{document}